\documentclass[11pt]{amsproc}
\usepackage{ytableau,graphicx,xypic,hyperref,accents}
\makeatletter
\DeclareRobustCommand{\cev}[1]{%
  \mathpalette\do@cev{#1}%
}
\newcommand{\do@cev}[2]{%
  \fix@cev{#1}{+}%
  \reflectbox{$\m@th#1\vec{\reflectbox{$\fix@cev{#1}{-}\m@th#1#2\fix@cev{#1}{+}$}}$}%
  \fix@cev{#1}{-}%
}
\newcommand{\fix@cev}[2]{%
  \ifx#1\displaystyle
    \mkern#23mu
  \else
    \ifx#1\textstyle
      \mkern#23mu
    \else
      \ifx#1\scriptstyle
        \mkern#22mu
      \else
        \mkern#22mu
      \fi
    \fi
  \fi
}

\makeatother
\xyoption{all}
\title{An Introduction to Schur Polynomials}
\author{Amritanshu Prasad}
\address{The Institute of Mathematical Sciences, Chennai.}
\address{Homi Bhabha National Institute, Mumbai.}
\email{amri@imsc.res.in}
\newtheorem{theorem}{Theorem}[subsection]
\newtheorem{lemma}[theorem]{Lemma}
\newtheorem{corollary}[theorem]{Corollary}

\theoremstyle{definition}
\newtheorem{definition}[theorem]{Definition}
\theoremstyle{example}
\newtheorem{example}[theorem]{Example}
\newtheorem{exercise}[theorem]{Exercise}

\newcommand{\ev}{\textup{ev}}
\newcommand{\rins}{\iota}
\newcommand{\ins}{\textup{INSERT}}
\newcommand{\del}{\textup{DELETE}}
\newcommand{\rdel}{\partial}
\newcommand{\pl}{\textup{Pl}}
\newcommand{\supp}{\textup{supp}}
\newcommand{\wt}{\textup{wt}}
\newcommand{\shape}{\textup{shape}}
\newcommand{\abc}{\textup{Abc}}
\newcommand{\Tab}{\textup{Tab}}
\newcommand{\rsk}{\textup{RSK}}
\newcommand{\bur}{\textup{BUR}}
\renewcommand{\circled}[1]{\raisebox{.5pt}{\textcircled{\raisebox{-.9pt} {#1}}}}
\DeclareMathOperator{\tab}{Tab}
\begin{document}
\maketitle
\setcounter{tocdepth}{2}
\tableofcontents
\subsection{Symmetric Polynomials}
\label{sec:symmetric-functions}
Consider polynomials in $n$ variables \linebreak $x_1,\dotsc,x_n$ having integer coefficients.
Given a multiindex $\alpha=(\alpha_1,\dotsc, \alpha_n)$, let $x^\alpha$ denote the monomial $x_1^{\alpha_1}\dotsb x_n^{\alpha_n}$.
A \emph{symmetric polynomial} is a polynomial of the form
\begin{displaymath}
  f(x_1,\dotsc, x_n) = \sum_{\alpha} c_\alpha x^\alpha, \text{ with } c_\alpha\in \mathbf Z,
\end{displaymath}
where, for any permutation $w\in S_n$,
\begin{displaymath}
  c_{(\alpha_1,\dotsc,\alpha_n)} = c_{(\alpha_{w(1)},\dotsc,\alpha_{w(n)})}.
\end{displaymath}
The integer partition $\lambda$ obtained by sorting the coordinates of $\alpha$  is called the \emph{shape} of $\alpha$, denoted $\lambda(\alpha)$.
The most obvious example of a symmetric polynomial in $n$ variables is the \emph{monomial symmetric polynomial}, defined for each integer partition $\lambda$:
\begin{displaymath}
  m_\lambda = \sum_{\lambda(\alpha) = \lambda} x^\alpha.
\end{displaymath}
Note that $m_\lambda$ is homogeneous of degree $|\lambda|$ (the sum of the parts of $\lambda$).
\begin{exercise}
  Take $n=4$. Compute the monomial symmetric polynomials $m_{(3)}$, $m_{(2,1)}$, and $m_{(1^3)}$.
\end{exercise}
\begin{theorem}
The polynomials $m_\lambda(x_1,\dotsc,x_n)$, as $\lambda$ runs over all the integer partitions of $d$, form a basis for the space of homogeneous symmetric polynomials of degree $d$ in $n$ variables.
\end{theorem}
\subsection{Complete and Elementary Symmetric Polynomials}
\label{sec:compl-elem-symm}
Recall that the coefficients of a polynomial are symmetric polynomials in its roots:
\begin{multline}
  \label{eq:elem-id}
  (t-x_1)(t-x_2)\dotsb (t-x_n) \\= t^n - e_1(x_1,\dotsc, x_n)t^{n-1} + \dotsb + (-1)^n e_n(x_1,\dotsc, x_n),
\end{multline}
where the expression $e_i(x_1,\dotsc, x_n)$ in the coefficient of $t^{n-i}$ is given by:
\begin{equation}
  \label{eq:elem}
  e_i(x_1,\dotsc, x_n) = \sum_{1\leq j_1<\dotsb<j_i\leq n} x_{j_1}x_{j_2}\dotsb x_{j_i}.
\end{equation}
The polynomial $e_i$ is called the $i$th \emph{elementary symmetric polynomial}.
By convention, $e_i(x_1,\dotsc,x_n)=0$, for $i>n$.

The identity (\ref{eq:elem-id}) can be written more elegantly as:
\begin{displaymath}
  (1+t x_1) \dotsb (1+tx_n) = \sum_{i=0}^\infty e_i(x_1,\dotsb, x_n)t^i.
\end{displaymath}

Dually, the \emph{complete symmetric polynomials} are defined by the formal identity:
\begin{displaymath}
  \frac 1{(1-x_1t)\dotsb (1-x_nt)} = \sum_{i=0}^\infty h_i(x_1,\dotsb, x_n)t^i.
\end{displaymath}
\begin{example}
  In three variables:
  \begin{align*}
    e_2(x_1,x_2,x_3) & = x_1x_2 + x_1x_3 + x_2x_3,\\
    h_2(x_1,x_2,x_3) & = x_1^2 + x_1x_2 + x_1x_3 + x_2^2 + x_2x_3 + x_3^2.
  \end{align*}
\end{example}
\begin{exercise}
  Show that
  \begin{displaymath}
    h_i(x_1,\dotsc, x_n) = \sum_{1\leq j_1\leq \dotsb \leq j_i\leq n} x_{j_1}\dotsb x_{j_i},
  \end{displaymath}
  and that
  \begin{displaymath}
    e_i(x_1,\dotsc, x_n) = \sum_{1\leq j_1< \dotsb < j_i\leq n} x_{j_1}\dotsb x_{j_i}.
  \end{displaymath}
\end{exercise}
More generally, for any integer partition $\lambda=(\lambda_1,\dotsc, \lambda_l)$, define:
\begin{align*}
  h_\lambda &= h_{\lambda_1} h_{\lambda_2}\dotsb h_{\lambda_l},\\
  e_\lambda &= e_{\lambda_1} e_{\lambda_2}\dotsb e_{\lambda_l}.
\end{align*}
\begin{theorem}
  Given partitions $\lambda=(\lambda_1,\dotsc,\lambda_l)$ and $\mu=(\mu_1,\dotsb, \mu_m)$ of $d$, let $M_{\lambda\mu}$ denote the number of matrices $(a_{ij})$ with non-negative integer entries whose $i$th row sums to $\lambda_i$ for each $i$, and whose $j$th column sums to $\mu_j$ for each $j$.
  Then
  \begin{displaymath}
    h_\lambda = \sum_\mu M_{\lambda\mu} m_\mu.
  \end{displaymath}
  Dually, let $N_{\lambda\mu}$ denote the number of integer matrices $(a_{ij})$ with entries $0$ or $1$, whose $i$th row sums to $\lambda_i$ for each $i$, and whose $j$th column sums to $\mu_j$ for each $j$.
  Then
  \begin{displaymath}
    e_\lambda = \sum_\mu N_{\lambda\mu} m_\mu.
  \end{displaymath}
\end{theorem}
\begin{proof}
  To prove the second identity involving elementary symmetric polynomials, note that a monomial in the expansion of
  \begin{displaymath}
    e_\lambda = \prod_{i=1}^l \sum_{j_1<\dotsb<j_{\lambda_i}}x_{j_1}\dotsb x_{j_{\lambda_i}}
  \end{displaymath}
  is a product of summands, one chosen from each of the $l$ factors.
  Construct an $l\times m$ matrix $(a_{ij})$ corresponding to such a choice as follows:
  if the summand $x_{j_1}\dotsb x_{j_{\lambda_i}}$ is chosen from the $i$th factor, then set the entries $a_{i,j_1},\dotsc, a_{i, j_{\lambda_j}}$ to be $1$ (the remaining entries of the $i$th row are $0$).
  Clearly the $i$th row of such a matrix sums to $\lambda_i$.
  The monomial corresponding to this choice is $x^\mu$ if, for each $j$, the the number of $i$ for which $x_j$ appears in the monomial corresponding to the $i$th row is $\mu_j$. This is just the sum of the $j$th column of the matrix $(a_{ij})$.
  It follows that the coefficient of $x^\mu$, and hence the coefficient of $m_\mu$ in the expansion of $e_\lambda$ in the basis of monomial symmetric polynomials of degree $n$, is $N_{\lambda\mu}$.

A similar proof can be given for the first identity involving complete symmetric polynomials. The only difference is that variables may be repeated in the monomials that appear in $h_i$. Counting the number of repetitions (instead of just recording $0$ or $1$) gives non-negative integer matrices.
\end{proof}
\subsection{Alternating Polynomials}
\label{sec:alt-poly}
An \emph{alternating polynomial} in $x_1,\dotsc, x_n$ is of the form:
\begin{equation}
  \label{eq:alt-form}
  f(x_1,\dotsc,x_n) = \sum_{\alpha} c_\alpha x_\alpha,
\end{equation}
where, $c_{(\alpha_{w(1)},\dotsc,\alpha_{w(n)})} = \epsilon(w)c_{(\alpha_1,\dotsc,\alpha_n)}$ for every multiindex $\alpha$ as in Section~\ref{sec:symmetric-functions}, and every permutation $w\in S_n$.
Here $\epsilon:S_n\to \{\pm 1\}$ denotes the sign function.
Equivalently, an alternating polynomial is one whose sign is reversed upon the interchange of any two variables.
\begin{exercise}
  If $\alpha$ is a multiindex with $\alpha_i=\alpha_j$ for some $i\neq j$, then $c_\alpha = 0$.
\end{exercise}
In particular, every monomial in an alternating polynomial must be composed of distinct powers.
Moreover, the polynomial is completely determined by the coefficients with strictly decreasing multiindices, namely, multiindices of the form $c_\alpha$, where $\alpha=(\alpha_1,\dotsc,\alpha_n)$ with $\alpha_1>\dotsb>\alpha_n$.
\begin{exercise}
  Let $\delta=(n-1,n-2,\dotsc,1, 0)$.
  Given an integer partition with at most $n$ parts, we will pad it with $0$'s so that it can be regarded as a weakly decreasing multiindex of length $n$.
  Then $\lambda\mapsto \lambda+\delta$ is a bijection from the set of integer partitions with at most $n$ parts onto the set of strictly decreasing multiindices.
\end{exercise}
\begin{example}
  Let $\lambda = (\lambda_1,\dotsc, \lambda_n)$ be a weakly decreasing multiindex.
  The \emph{alternant} corresponding to $\lambda$, which is defined as:
  \begin{displaymath}
    a_{\lambda+\delta} = \det(x_i^{\lambda_j + n - j})
  \end{displaymath}
  is alternating, with unique strictly decreasing monomial $x^{\lambda+\delta}$.
\end{example}
\begin{exercise}
  \label{exercise:alt-basis}
  The alternating polynomial of the form \textup{(\ref{eq:alt-form})} is equal to  \begin{displaymath}
    \sum_{\lambda} c_{\lambda+\delta} a_{\lambda+\delta},
  \end{displaymath}
  the sum being over all weakly decreasing multiindices $\lambda$.
\end{exercise}
\subsection{Interpretation of Alternants with Labeled Abaci}
\label{sec:abaci}
A labeled abacus with $n$ beads is a word $w=(w_k; k\geq 0)$ with letters $w_i\in \{0,\dotsc,n\}$ such that the subword of non-zero letters is a permutation of $1,2,\dotsc,n$.
The sign $\epsilon(w)$ of the abacus is the sign of this permutation, the support is the set $\supp(w)=\{k\mid w_k>0\}$, and the weight is defined as:
\begin{displaymath}
  \wt(w) = \prod_{k\in \supp(w)} x_{w_k}^k.
\end{displaymath}
The shape of the abacus, $\shape(w)$ is the unique partition $\lambda$ such that the components of $\lambda+\delta$ form the support of $w$.
\begin{example}
  Consider the labeled abacus $w=510032046000\dotsb$.
  Its underlying permutation is $513246$, which has sign $-1$, so $\epsilon(w)=-1$.
  Also, $\supp(w)=\{0,1,4,5,7,8\}$, $\shape(w)=(3,3,2,2)$ (indeed, $(3,3,2,2,0,0)+(5,4,3,2,1,0)=(8,7,5,4,1,0)$) and $\wt(w)=x_5^0x_1^1x_3^4x_2^5x_4^7x_6^8$.
  We visualize the abacus $w$ as a configuration of beads on a single runner, with possible positions of beads numbered $1, 2, 3, \dotsc$.
  If $w_k=i$ where $i>0$, then a bead labeled $i$ is placed in position $k$ on the runner.
  If $w_k=0$, then the position $k$ is unoccupied.
  In the running example the visualization is:
  \begin{displaymath}
    \begin{tabular}{*{10}{p{0.7cm}}}
      0 & 1 & 2 & 3 & 4 & 5 & 6 & 7 & 8 & 9\\
      \circled{5} & \circled{1} & $\bullet$ & $\bullet$ & \circled{3} & \circled{2} & $\bullet$ & \circled{4} & \circled{6} & $\bullet$
    \end{tabular}
  \end{displaymath}
  The first row shows the positions $k=0,1,\dotsc$ on the runner and the second row shows the beads.
\end{example}
\begin{theorem}
  \label{theorem:abacus-alt}
  For every partition $\lambda$ the alternant in $n$  variables,
  \begin{displaymath}
    a_{\lambda+\delta} = (-1)^{\lfloor n/2\rfloor} \sum_w \epsilon(w)\wt(w),
  \end{displaymath}
  the sum being over all labeled abaci with $n$ beads and shape $\lambda$.
\end{theorem}
\begin{proof}
  The theorem follows from the expansion of the determinant.
\end{proof}
\subsection{Cauchy's Bialternant Form of a Schur Polynomial}
\label{sec:cauchys-bialt-form}
The lowest degree polynomial of the form $a_{\lambda+\delta}$ arises when $\lambda=0$; $a_\delta$ is the Vandermonde determinant:
\begin{displaymath}
  a_\delta = \prod_{1\leq i<j\leq n}(x_i-x_j).
\end{displaymath}
\begin{exercise}
  Show that, for every weakly decreasing multiindex $\lambda$, $a_{\lambda+\delta}$ is divisible by $a_\delta$ in the ring of polynomials in $x_1,\dotsc,x_n$.
\end{exercise}
\begin{exercise}
  \label{exercise:vandermonde-iso}
  Show that $f\mapsto fa_\delta$ is an isomorphism of the space of homogeneous symmetric polynomials in $x_1,\dotsc, x_n$ of degree $d$ onto the space of homogeneous alternating polynomials of degree $d + \binom n2$.
\end{exercise}
This motivates the historically oldest definition of Schur polynomials---\emph{Cauchy's bialternant formula}:
\begin{equation}
  \label{eq:schur}
  s_\lambda(x_1,\dotsc,x_n) = a_{\lambda+\delta}/a_\delta,
\end{equation}
for any partition $\lambda$ with at most $n$ parts.
If $\lambda$ has more than $n$ parts, set $s_\lambda(x_1,\dotsc,x_n) =0$.
This is clearly a symmetric polynomial of degree $|\lambda|$.
\begin{theorem}
  As $\lambda$ runs over all integer partitions of $d$ with at most $n$ parts, the Schur polynomials $s_\lambda(x_1,\dotsc,x_n)$ form a basis of the space of all homogeneous symmetric polynomials in $x_1,\dotsc,x_n$ of degree $d$.
\end{theorem}
\begin{proof}
  This follows from Exercises~\ref{exercise:alt-basis} and~\ref{exercise:vandermonde-iso}.
\end{proof}
\begin{exercise}
  [Stability of Schur polynomials]
  Show that substituting $x_n=0$ in the Schur polynomial $s_\lambda(x_1,\dotsc, x_n)$ with $n$ variables gives the corresponding Schur polynomial $s_\lambda(x_1,\dotsc,x_{n-1})$ with $n-1$ variables.
\end{exercise}
\subsection{Pieri's rule}
\label{sec:pieri}
The set of integer partitions is endowed with the \emph{containment order}.
We say that a partition $\lambda=(\lambda_1,\dotsc,\lambda_l)$ \emph{contains} a partition $\mu=(\mu_1,\dotsc, \mu_m)$ if $l \geq m$, and $\lambda_i\geq \mu_i$ for every $i=1,\dotsb, m$.
We write $\lambda\supset\mu$ or $\mu \subset \lambda$.
Recall that the Young diagram of the partition $\lambda$ is the set of points 
\begin{displaymath}
\{(i, j)\mid 1\leq i\leq l,\; 1\leq j\leq \lambda_i\}.
\end{displaymath}
Visually, each node $(i,j)$ of the Young diagram is replaced by a box, and the box corresponding to $(i,j)$ is placed in the $i$th row and $j$th column (matrix notation).
Thus, the Young diagram of $\lambda=(6, 5, 3, 3)$ is:
\ytableausetup{smalltableaux}
\begin{displaymath}
  \ydiagram{6,5,3,3}
\end{displaymath}
The containment of partitions is nothing but the containment relation on their Young diagrams.
Henceforth, for a partition $\lambda$, the symbol $\lambda$ will also be used to refer to its Young diagram.

A \emph{skew-shape} is a difference of Young diagrams $\lambda \setminus \mu$, where $\lambda \supset \mu$.
Write $\lambda/\mu$ for this skew-shape.
A skew-shape is called a \emph{horizontal strip} (respectively, a \emph{vertical strip}) if it has at most one box in each vertical column (respectively, horizontal row).
\begin{theorem}
  For every partition $\lambda$, and every positive integer $k$,
  \begin{displaymath}
    s_\lambda h_k = \sum_\mu s_\mu,
  \end{displaymath}
  where the sum runs over all partitions $\mu\supset\lambda$ such that $\mu/\lambda$ is a horizontal strip of size $k$.
  Dually,
  \begin{displaymath}
    s_\lambda e_k = \sum_\mu s_\mu,
  \end{displaymath}
  where the sum runs over all partitions $\mu\supset\lambda$ such that $\mu/\lambda$ is a vertical strip of size $k$.
\end{theorem}
\begin{proof}
  Let $\abc(\lambda)$ denote the set of all $n$-bead labeled abaci (see Section~\ref{sec:abaci}) of shape $\lambda$.
  Let $M(n,k)$ denote the set of all vectors $\alpha=(\alpha_1,\dots,\alpha_n)$ with non-negative integer coordinates and sum $k$.
  Set $\wt(\alpha)=x_1^{\alpha_1}\dotsb x_n^{\alpha_n}$.
  Using the abacus interpretation of alternants (Theorem~\ref{theorem:abacus-alt}), the first identity is equivalent to showing:
  \begin{displaymath}
    \sum_{w\in \abc(\lambda)} \epsilon(w)\sum_{\alpha\in M(n,k)} \wt(\alpha) = \sum_\mu \sum_{w\in \abc(\mu)} \epsilon(w) \wt(w),
  \end{displaymath}
  the sum on the right being over all partitions $\mu\supset\lambda$ such that $\mu/\lambda$ is a horizontal strip.
  We will define an involution $I$ on the $\abc(\lambda)\times M(n,k)$ whose fixed points correspond to elements of
\begin{displaymath}
  \coprod_{\mu/\lambda \text{ is a horiz. strip of size $k$}}\abc(\mu)\times M(n,k)
\end{displaymath}
under a bijection that preserves weights and signs, and such that if $I(w, \alpha)=(w', \alpha')$ then $\wt(w)\wt(\alpha)=\wt(w')\wt(\alpha')$ and $\epsilon(w') = -\epsilon(w)$.
  Then all terms on the left hand side, except for those which do not correspond to fixed points, will cancel, and the surviving terms will give the right hand side.

  To construct $I$, scan the abacus from left to right.
  Upon encountering a bead numbered $j$, move the bead $\alpha_j$ steps to the right, one step at a time.
  If this process completes without this bead colliding with another bead, $(w,\alpha)$ is a fixed point of $I$.
  The new abacus $w^*$ has $\epsilon(w^*)=\epsilon(w)$ (the underlying permutation remains unchanged), and $\shape(w^*)/\shape(w)$ is a horizontal strip of size $k$.

  However, suppose a collision does occur, say the first collision is when bead $j$ hits bead $k$ that is located $p\leq \alpha_j$ places to the right of its initial position.
  Define $I(w,\alpha) = (w',\alpha')$, where $w'$ is $w$ with the beads $i$ and $j$ interchanged, $\alpha'_j=\alpha_j-p$, $\alpha'_k=\alpha_k+p$ and all other coordinates of $\alpha$ and $\alpha'$ are equal.
  Clearly $w'$ has the opposite sign from $w$, and $\wt(w)\wt(\alpha)=\wt(w')\wt(\alpha')$.
  It is not hard to see that $I(w',\alpha')=(w,\alpha)$.
  \begin{example}
    \label{example:bead-h}
    Let $n=6$, $\lambda=(3,3,2,2,0,0)$, $k=3$, and
    \begin{displaymath}
      (w, \alpha) = (51003204600\dotsb, (2, 1, 0, 0, 0, 0)). 
    \end{displaymath}
    Scanning the abacus from left to right, the first bead to be moved is numbered $1$.
    It can be moved $2$ places to the right without any collisions.
    After that the bead numbered $2$ can be moved $1$ place to the right, again without collisions.
    So $(w,\alpha)$ is a fixed point for $I$.
    The new abacus $50013024600\dotsb$ has shape $(3,3,3,2,2,0)$ obtained by adding a horizontal $3$-strip to $(3, 3, 2, 2, 0, 0)$.

    On the other hand, if $\alpha=(1, 1, 1, 0, 0, 0)$, then the first collision is of the bead numbered $3$ with the bead numbered $2$.
    Interchanging the beads numbered $2$ and $3$, and modifying the weights as prescribed gives $I(w,\alpha) = (51002304600\dotsb, (1,2,0,0,0,0))$.
  \end{example}

  Let $N(n,k)$ denote the set of vectors $\alpha=(\alpha_1,\dotsc,\alpha_n)$ such that $\alpha_i\in \{0,1\}$ for each $i$, and $\alpha_1+\dotsb+\alpha_n=k$.
  In terms of Abaci, the second Pieri rule becomes:
  \begin{displaymath}
    \sum_{w\in \abc(\lambda)} \epsilon(w)\sum_{\alpha\in N(n,k)} \wt(\alpha) = \sum_\mu \sum_{w\in \abc(\mu)} \epsilon(w) \wt(w),
  \end{displaymath}
  where $\mu$ runs over all partition such that $\mu/\lambda$ is a vertical strip of size $k$.
  
  We construct an involution $I$ on $\abc(\lambda)\times N(n,k)$ as follows: scan the abacus from \emph{right to left}.
  Upon encountering a bead numbered $j$, if $\alpha_j=1$, try to move the bead one step to the right.
  If this process completes without collisions, then $(w,\alpha)$ is a fixed point of $I$.
  Otherwise, if the first collision occurs with bead numbered $j$ colliding with bead numbered $k$, then define $w'$ to be $w$ with beads $j$ and $k$ interchanged.
  Also, since the $k$th bead was adjacent to the $j$th bead, it could not have been moved in its turn.
  So $\alpha_k=0$.
  Let $w'$ be the abacus obtained by interchanging beads numbered $k$ and $j$ in $w$, 
  let $\alpha'$ be obtained by interchanging $\alpha_k$ and $\alpha_j$ in $\alpha$, and set $I(w,\alpha)=(w',\alpha')$.
\end{proof}
  \begin{example}
  The pair $(51003204600\dotsb,(1,1,1,0,0,0))$ is a fixed point for $I$, and the shifted abacus is $(50100324600\dotsb)$ of shape $(3,3,3,3,1,0)$.
  On the other hand 
  \begin{displaymath}
    I(51003204600\dotsb,(0,0,1,0,1,1))=(51002304600\dotsb, (0,1,0,0,1,1)).
  \end{displaymath}
\end{example}
The following is a special case of Pieri's rule:
\begin{corollary}
  For every positive integer $k$,
  \begin{displaymath}
    s_{(k)} = h_k, \text{ and } s_{(1^k)} = e_k.
  \end{displaymath}
\end{corollary}
\begin{exercise}
  \label{exercise:hook-schur}
  Use Pieri's rule to show that:
  \begin{displaymath}
    h_ke_l = s_{(k, 1^l)} + s_{(k+1, 1^{l-1})}.
  \end{displaymath}
  Conclude that
  \begin{displaymath}
    s_{(j+1,1^k)} = \sum_{l=0}^k (-1)^l h_{j+l+1}e_{k-l}.
  \end{displaymath}
\end{exercise}
\subsection{Schur to Complete and Elementary via Tableaux}
Pieri's rule allows us to compute the complete and elementary symmetric polynomials $h_\lambda$ and $e_\lambda$ in terms of Schur polynomials.
\begin{example}
  \label{example:expansion-e}
  Repeated application of Pieri's rule gives an expansion of $e_{(2, 2, 1)} = e_2e_2e_1$ as:
  \begin{displaymath}
    \resizebox{\textwidth}{!}{
      \xymatrix@R-15pt{
        &&&e_2 e_2 e_1\ar[d]&&&\\
        &&&s_{(1^2)} e_2 e_1\ar[dll]\ar[d]\ar[drr]&&\\
        &s_{(2^2)}e_1\ar[dl]\ar[d] && s_{(2, 1^2)}e_1\ar[dl]\ar[d]\ar[dr] && s_{(1^4)}e_1 \ar[d]\ar[dr]&\\
        s_{(3, 2)} & s_{(2^2, 1)} & s_{(3, 1^2)} & s_{(2^2, 1)} & s_{(2, 1^3)} & s_{(2, 1^3)} & s_{(1^5)}
      }
    }
  \end{displaymath}
  giving:
  \begin{displaymath}
    e_{(2^2,1)} = s_{(3,2)} + 2s_{(2^2,1)} + s_{(3,1^2)} + 2s_{(2,1^3)} + s_{(1^5)}.
  \end{displaymath}
  The steps going from the first line of the above calculation to each term of the last line can be recorded by putting numbers into Young diagrams:
  \begin{displaymath}
    \resizebox{0.75\textwidth}{!}{
      \xymatrix@R-15pt{
        &&&\emptyset\ar[d]&&&\\
        &&&\ytableaushort{1,1}\ar[dll]\ar[d]\ar[drr]&&\\
        &\ytableaushort{12,12}\ar[dl]\ar[d] && \ytableaushort{12,1,2}\ar[dl]\ar[d]\ar[dr] && \ytableaushort{1,1,2,2}\ar[d]\ar[dr]&\\
        \ytableaushort{123,12} & \ytableaushort{12,12,3} & \ytableaushort{123,1,2} & \ytableaushort{12,13,2} & \ytableaushort{12,1,2,3} & \ytableaushort{13,1,2,2} & \ytableaushort{1,1,2,2,3}
      }
    }
  \end{displaymath}
  The boxes in the vertical strip added at the $i$th stage are filled with $i$.
\end{example}
\begin{example}
  \label{example:expansion-h}
  Repeated application of Pieri's rule gives an expansion of $h_{(2, 2, 1)} = h_2h_2h_1$ as:
  \begin{displaymath}
    \resizebox{\textwidth}{!}{
      \xymatrix@R-15pt{
        &&&h_2 h_2 h_1\ar[d]&&&\\
        &&&s_{(2)} h_2 h_1\ar[dll]\ar[d]\ar[drr]&&\\
        &s_{(4)}h_1\ar[dl]\ar[d] && s_{(3,1)}h_1\ar[dl]\ar[d]\ar[dr] && s_{(3,2)}h_1 \ar[d]\ar[dr]&\\
        s_{(5)} & s_{(4, 1)} & s_{(4,1)} & s_{(3, 2)} & s_{(3, 1^2)} & s_{(3,2)} & s_{(2,2,1)}
      }
    }
  \end{displaymath}
  giving:
  \begin{displaymath}
    h_{(2^2,1)} = s_{(5)} + 2s_{(4,1)} + 2s_{(3,2)} + s_{(3,1^2)} + s_{(2,2,1)}.
  \end{displaymath}
  The steps going from the first line of the above calculation to each term of the last line can be recorded by putting numbers into Young diagrams:
  \ytableausetup{smalltableaux}
  \begin{displaymath}
    \resizebox{\textwidth}{!}{
      \xymatrix@R-15pt{
        &&&\emptyset\ar[d]&&&\\
        &&&\ytableaushort{11}\ar[dll]\ar[d]\ar[drr]&&\\
        &\ytableaushort{1122}\ar[dl]\ar[d] && \ytableaushort{112,2}\ar[dl]\ar[d]\ar[dr] && \ytableaushort{11,22}\ar[d]\ar[dr]&\\
        \ytableaushort{11223} & \ytableaushort{1122,3} & \ytableaushort{1123,2} & \ytableaushort{112,23} & \ytableaushort{112,2,3} & \ytableaushort{113,22} & \ytableaushort{11,22,3}
      }
    }
  \end{displaymath}
  The boxes in the horizontal strip added at the $i$th stage are filled with $i$.
\end{example}
\begin{definition}
  [Semistandard tableau]
  A semistandard tableau of shape $\lambda=(\lambda_1,\dotsc,\lambda_l)$ and type $\mu=(\mu_1,\dotsc,\mu_m)$ is the Young diagram of $\lambda$ filled with numbers $1,\dotsc, m$ such that the number $i$ appears $\mu_i$ times, the numbers weakly increase along rows, and strictly increase along columns.
\end{definition}
\begin{exercise}
  Semistandard tableaux of shape $\lambda$ and type $\mu$ correspond to chains of integer partitions
  \begin{displaymath}
    \emptyset = \lambda^{(0)} \subset \lambda^{(1)}\subset \lambda^{(2)} \subset \dotsb \subset \lambda^{(m)} = \lambda
  \end{displaymath}
  where $\lambda^{(i)}/\lambda^{(i-1)}$ is a horizontal strip of size $\mu_i$.
\end{exercise}
\begin{example}
  The semistandard tableau of type $(3,2)$ and type $(2,2,1)$ are $\ytableaushort{112,23}$ and $\ytableaushort{113,22}$.
  They correspond to the chains:
  \begin{displaymath}
    \ydiagram{2}\subset \ydiagram{3,1}\subset \ydiagram{3,2} \text{ and } \ydiagram{2}\subset \ydiagram{2,2}\subset \ydiagram{3,2},
  \end{displaymath}
  respectively.
  As illustrated in Example~\ref{example:expansion-h}, the coefficient of $s_{(3,2)}$ in the complete symmetric polynomial $h_{(2,2,1)}$ is the number of semistandard tableau of shape $(3,2)$ and type $(2,2,1)$.
\end{example}
\begin{definition}
  [Kostka number]
  Given two partitions $\lambda$ and $\mu$, the Kostka number $K_{\lambda\mu}$ is the number of semistandard tableaux of shape $\lambda$ and type $\mu$.
\end{definition}
\begin{exercise}
  \label{exercise:unit-kostka}
  For every partition $\lambda$, show that $K_{\lambda\lambda}=1$.
\end{exercise}
\begin{definition}
  [$f$-number]
  The $f$-number of a partition $\lambda$ of $n$ is defined to be the Kostka number $K_{\lambda,(1^n)}$, and is denoted $f_\lambda$.
\end{definition}
\begin{exercise}
  For a partition $\lambda$, let $\lambda^-$ denote the set of all partitions whose Young diagram can be obtained by removing one box from the Young diagram of $\lambda$.
  For each $\lambda\neq \emptyset$, show that $f_\lambda = \sum_{\mu\in \lambda^-} f_\mu$.
\end{exercise}
\begin{exercise}
  A hook is a partition of the form $h(a,b)=(a+1,1^b)$.
  Show that $f_{h(a,b)}=\binom{a+b}a$.
\end{exercise}
In order to understand the expansion of elementary symmetric polynomials we would need a variant of semistandard tableaux, one where the difference between successive shapes are vertical strips, rather than horizontal strips.
However, it has become common practice to \emph{conjugate} partitions instead:
\begin{definition}
  [Conjugate of a partition]
  The \emph{conjugate} of a partition $\lambda$ is the partition $\lambda'$ whose Young diagram is given by:
  \begin{displaymath}
    \lambda' = \{(j,i)\mid (i, j)\in \lambda\}.
  \end{displaymath}
  In other words, the Young diagram of $\lambda'$ is the reflection of the Young diagram of $\lambda$ about its principal diagonal.
\end{definition}
Clearly $\lambda\mapsto\lambda'$ is an involution.
For example, if $\lambda=(2,2,1)$, then $\lambda'=(3,2)$.
\begin{exercise}
Semistandard tableaux of shape $\lambda'$ and type $\mu$ correspond to chains of integer partitions
  \begin{displaymath}
    \emptyset = \lambda^{(0)} \subset \lambda^{(1)}\subset \lambda^{(2)} \subset \dotsb \subset \lambda^{(m)} = \lambda
  \end{displaymath}
  where $\lambda^{(i)}/\lambda^{(i-1)}$ is a \emph{vertical} strip of size $\mu_i$.  
\end{exercise}
The method for computing elementary and complete polynomials from Schur polynomials illustrated in Examples~\ref{example:expansion-e} and \ref{example:expansion-h} can be expressed as follows:
\begin{theorem}
  \label{theorem:schur-to-eh}
  The expansion of complete symmetric polynomials in terms of Schur polynomials is given by:
  \begin{displaymath}
    h_\mu = \sum_\lambda K_{\lambda\mu}s_\lambda.
  \end{displaymath}
  Dually, the extension of elementary symmetric polynomials in terms of Schur polynomials is given by:
  \begin{displaymath}
    e_\mu = \sum_\lambda K_{\lambda'\mu}s_\lambda.
  \end{displaymath}
\end{theorem}
\subsection{Triangularity of Kostka Numbers}
\label{sec:triang-kostka-numb}
If partitions $\lambda=(\lambda_1,\dotsc,\lambda_l)$ and $\mu=(\mu_1,\dotsc,\mu_m)$ have $K_{\lambda\mu}>0$,
then there exists a semistandard tableau $t$ of shape $\lambda$ and type $\mu$.
Since the columns of $t$ are strictly increasing, all the $1$'s in $t$ must occur in its first row, so $\lambda_1\geq \mu_1$.
Also, all the $2$'s must occur in the first two rows (along with all the $1$'s), so $\lambda_1+\lambda_2\geq \mu_1+\mu_2$.
More generally, all the numbers $1,\dotsc, i$ for $i=1,\dotsc,m$ should occur in the first $i$ rows of $t$.
We have:
\begin{equation}
  \label{eq:dominance}
  \lambda_1+\dotsb + \lambda_i \geq \mu_1+\dotsb+\mu_i \text{ for } i=1,\dotsc, m.
\end{equation}
\begin{definition}
  The integer partition $\lambda$ \emph{dominates} the integer partition $\mu$ if $|\lambda|=|\mu|$ and (\ref{eq:dominance}) holds for $i=1,\dotsc,m$.
  When this happens write $\lambda\rhd \mu$.
  This relation defines a partial order on the set of all integer partitions of $n$ for any non-negative integer $n$.
\end{definition}
\begin{exercise}
  Show that $(n)$ is maximal and $(1^n)$ is minimal among all the integer partitions of $n$.
  What is the smallest integer $n$ for which the dominance order on partitions of $n$ is not a linear order?
\end{exercise}
\begin{theorem}
  [Triangularity of Kostka Numbers]
  \label{theorem:Kostka-triangularity}
  Given partitions $\lambda$ and $\mu$ of an integer $n$, $K_{\lambda\mu}>0$ if and only if $\lambda\rhd\mu$.
\end{theorem}
\begin{proof}
  It only remains to construct, whenever $\lambda\rhd\mu$, a semistandard tableau of shape $\lambda$ and type $\mu$.
  In order to understand the algorithm that follows, it is helpful to keep in mind Example~\ref{example:can-tab} below.
  Since $\lambda\rhd\mu$, $\lambda_1\geq \mu_1\geq \mu_m$.
  Therefore, the Young diagram of $\lambda$ has at least $\mu_m$ cells in its first row, or in other words, it has at least $\mu_m$ columns.
  Choose the largest integer $i$ for which $\lambda_i\geq\mu_m$.
  Fill the bottom-most box in the $\lambda_{i+1}$ leftmost columns with $m$.
  Also, from the $i$th row, fill the rightmost $\mu_m-\lambda_{i+1}$ boxes with $m$.
  The remaining (unfilled) boxes in the Young diagram of $\lambda$ now form the Young diagram of the partition
  \begin{displaymath}
  \eta=(\lambda_1,\dotsc,\lambda_{i-1}, \lambda_i - \mu_m + \lambda_{i+1}, \lambda_{i+2},\dotsc,\lambda_l), 
  \end{displaymath}
  a partition with $l-1$ parts.
  Writing $(\eta_1,\dotsc,\eta_{l-1})$ for the parts of $\eta$, note that, since the first $i-1$ parts of $\eta$ are the same as those of $\lambda$,
  we have:
  \begin{displaymath}
    \eta_1+\dotsb + \eta_j\geq \mu_1 + \dotsb + \mu_j
  \end{displaymath}
  for $j\leq i-1$.
  For $j\geq i$, we have
  \begin{align*}
    \eta_1 + \dotsb + \eta_j & = \lambda_1 + \dotsb + \lambda_{j+1} -\mu_m\\
                             & \geq \mu_1 + \dotsb + \mu_j + \mu_{j+1} - \mu_m\\
                             & \geq \mu_1 + \dotsb + \mu_j.
  \end{align*}
  It follows that $\eta\rhd (\mu_1,\dotsc,\mu_{m-1})$.
  Recursively applying this step to $\eta$ and $(\mu_1,\dotsc,\mu_{m-1})$ gives rise to a semistandard tableau of shape $\lambda$ and type $\mu$.
  The base case is where $\mu$ has only one part, in which case the dominance condition (\ref{eq:dominance}) implies that $\lambda=\mu$.
\end{proof}
\begin{example}
  \label{example:can-tab}
  Consider the case where $\lambda=(7,3,2)$ and $\mu=(4,4,4)$.
  Then the largest integer $i$ such that $\lambda_i\geq 4$ is $1$.
  Accordingly, we enter $3$ into the bottom-most boxes in the three leftmost columns, and also into one rightmost box in the first row:
  \begin{displaymath}
    \ytableaushort{{}{}{}{}{}{}3,{}{}3,33}
  \end{displaymath}
  We are left with the problem of finding a semistandard tableau of shape $(6,2)$ and type $(4,4)$.
  Recursively applying our process to this smaller problem gives:
  \begin{displaymath}
    \ytableaushort{{}{}{}{}223,223,33},
  \end{displaymath}
  and finally the desired tableau
  \begin{displaymath}
    \ytableaushort{1111223,223,33}.
  \end{displaymath}
\end{example}
\begin{theorem}
  The complete symmetric polynomials:
  \begin{displaymath}
    \{h_\mu \mid \text{$\mu$ is a partition of $d$ with at most $n$ parts}\}
  \end{displaymath}
  and the elementary symmetric polynomials:
  \begin{displaymath}
    \{e_\mu\mid \text{$\mu$ is a partition of $d$ with $\mu_1\leq n$}\}
  \end{displaymath}
  form bases of the space of homogeneous symmetric polynomials of degree $d$ in variables $x_1,\dotsc,x_n$.
\end{theorem}
\begin{proof}
  In view of the triangularity of Kostka numbers (Theorem~\ref{theorem:Kostka-triangularity}) and the fact that $K_{\lambda\lambda}=1$ (Exercise~\ref{exercise:unit-kostka}) the theorem follows from Theorem~\ref{theorem:schur-to-eh}.
\end{proof}
\subsection{Schensted's insertion algorithm}
\label{sec:schenst-insert-algor}
Let $t$ be a semistandard tableau, and $x$ be a positive integer.
Schensted's insertion algorithm is a method of inserting a box with the number $x$ into $t$, resulting in a new tableau $\ins(t, x)$.
Applied repeatedly, it gives a way to convert any word into a tableau.
This tableau succinctly expresses some combinatorial properties of the original word.

First consider the case where $t$ has a single row, with entries $a_1\leq \dotsb \leq a_k$.
Use $\emptyset$ to denote the empty word.
The algorithm $\rins$ takes as input the single row $t$ and a letter $x$, and returns a pair $(b, t')$, where $b'$ is either the empty word, or a single letter, and $t'$ is a row:
\begin{displaymath}
  \rins(a_1a_2\dotsb a_k, x) =
  \begin{cases}
    (\emptyset, a_1\dotsb a_k x) & \text{if $x\geq a_i$ for all $i$},\\
    (a_j, a_1\dotsb a_{j-1} x a_{j+1} \dotsb a_k) & \text{if } j = \min\{r \mid a_r > x\}.
  \end{cases}
\end{displaymath}
In the second case, one says that \emph{$x$ has been inserted into $t=a_1\dotsb a_k$, obtaining $t'=a_1\dotsb a_{j-1} x a_{j+1} \dotsb a_k$, and \textbf{bumping out} $a_j$}.
Also, it is notationally convenient to write $\rins(t, \emptyset) = (\emptyset, t)$ (when nothing is inserted, $t$ remains unchanged, and nothing is bumped out).

Now suppose $t$ is a tableau, with first row $r$.
Suppose that $\rins(r,x) = (y, r')$.
Recursively define $\ins(t,x)$ to be the tableau whose first row is $r'$, and remaining rows are the rows of $\ins(t',y)$,
where $t'$ is the tableau consisting of all but the first row of $t$.
\begin{example}
  \label{example:insertion}
  Consider the insertion of $3$ into the tableau:
  \begin{displaymath}
    t = \ytableaushort{13358,2466,358,4}.
  \end{displaymath}
  We have $\rins(13358,3) = (5,13338)$; $\rins(2466,5)=(6,2456)$; $\rins(358,6)=(8,356)$; $\rins(4,8)=(\emptyset,48)$.
  Thus, $\ins(t, 3)$ is the tableau:
  \begin{displaymath}
    \ytableaushort{13338,2456,356,48}.
  \end{displaymath}
In general, it is not possible to recover $t$ and $x$ from $\ins(t,x)$, even if we know $x$.
For example, the above tableau can be obtained by inserting $3$ into a different tableau:
\begin{displaymath}
  \ins\left(\vcenter{\hbox{\ytableaushort{13368,245,356,48}}}, 3\right) = \vcenter{\hbox{\ytableaushort{13338,2456,356,48}}}.
\end{displaymath}
\end{example}
Clearly, the shape of $\ins(t,x)$ is obtained by adding one box to the shape of $t$.
If we know the row $r$ into which the new box was added, and the value of $x$, then $t$ can be recovered from $\ins(t,x)$.
This recovery is based on the fact that $\rins$ can be inverted:
define
\begin{displaymath}
  \rdel(a, a_1a_2\dotsb a_k) = (a_1\dotsb a_{j-1}aa_{j+1}\dotsb a_k, a_j),
\end{displaymath}
where $j=k$ if $a_i\leq a$ for all $i=1,\dotsc, k$ and $j = \min\{i\mid a_{i+1}>a\}$.
To recover $t$ and $x$ from $s=\ins(t,x)$ and $r$ (the number of the row into which the new box was added), delete the last entry of the $r$th row of $s$, say $x_r$.
Let $u_{r-1}$ denote the $(r-1)$st row of $s$.
Suppose $\rdel(x_r,u_{r-1}) = (v_{r-1},x_{r-1})$, replace the $(r-1)$st row of $s$ with $v_{r-1}$.
Continue this process until $\rdel(x_2,u_1)=(v_1,x_1)$ is obtained and the first row of $s$ is replaced with $v_1$.
The tableau obtained at the end of this process is $t$, and $x=x_1$.
Write $\del(t,r)=(s,x)$.
The preceding discussion shows:
\begin{theorem}
  \label{theorem:del-ins}
  If $\del(t,r)=(s,x)$, then $\ins(s,x)=t$, and $\shape(t)$ is obtained from $\shape(s)$ by adding a cell to its $r$th row.
\end{theorem}
\begin{exercise}
  Verify Theorem~\ref{theorem:del-ins} for the insertions in Example~\ref{example:insertion}.
\end{exercise}
\subsection{Tableaux and Words}
\label{sec:tabl-assoc-word}
Let $L^*_n$ denote the concatenation monoid of all words in the alphabet $L_n=\{1,\dotsc,n\}$.
For any $w=a_1\dotsb a_k\in L^*_n$, Schensted's insertion algorithm allows us to associate a unique semistandard tableau $P(w)$ as follows:
\begin{itemize}
\item If $w=a$ has only one letter, then $P(a)$ is the single-cell tableau with entry $a$.
\item If $w=ua$, where $u\in L_n^*$ and $a\in \{1,\dotsc,n\}$, then \linebreak $P(w)=\ins(P(u), a)$.
\end{itemize}
\begin{example}
  \label{example:insertion}
  If $w=1374433254$, then $P(w)$ is the tableau:
  \begin{displaymath}
    \ytableaushort{12334,345,4,7}
  \end{displaymath}
\end{example}
\label{sec:tableau-as-words}
Given a semistandard tableau $t$, its reading word $w$ is defined to be the sequence of numbers obtained from reading its rows from left to right, starting with the bottom row, and moving up sequentially to the top row.
Since the first entry of each row is strictly smaller than the last entry of the row below it, the tableau $t$ can be recovered from $w$ by chopping it up into segments with a cut after each $a_i$ with $a_{i+1}<a_i$ (we say that $w$ has a descent at $i$). The resulting segments, taken from right to left, form the rows of $t$.
\begin{example}
  The reading word of the tableau $t$ formed at the end of Example~\ref{example:insertion} is:
  \begin{displaymath}
    w = 7434512334.
  \end{displaymath}
  The tableau $t$ is recovered by marking off the descents $w = 7|4|345|12334$, and then rearranging the segments into a tableau.
\end{example}
\begin{exercise}
  \label{exercise:tableau-word}
  Let $w$ denote the reading word of a tableau $t$.
  Show that $P(w)=t$.
\end{exercise}
Not every word comes from a semistandard tableau; for example the word $132$, when broken up at descents gives rise to $\ytableaushort{2,13}$.
We shall say that a word is a tableau if it is the reading word of a semistandard tableau.

Call the word $w=a_1\dotsb a_k$ a \emph{row} if $a_1\leq \dotsb \leq a_k$.
Call it a \emph{column} if $a_1>\dotsb>a_k$.
Write $x^w$ for the monomial $x_{a_1}x_{a_2}\dotsb x_{a_k}$.
\begin{exercise}
  \label{exercise:row-col}
  Show that, for every positive integer $i$,
  \begin{displaymath}
    h_k(x_1,\dotsc, x_n) = \sum_{\text{$w\in L_n^*$ is a row of length $k$}}x^w,
  \end{displaymath}
  and 
  \begin{displaymath}
    e_k(x_1,\dotsc, x_n) = \sum_{\text{$w\in L_n^*$ is a column of length $k$}}x^w.
  \end{displaymath}
\end{exercise}
If $w_1$ and $w_2$ are words, and $w_1w_2$ is their concatenation, then
\begin{displaymath}
  x^{w_1}x^{w_2} = x^{w_1w_2}.
\end{displaymath}
This gives rise to an algebra homomorphism called the \emph{evaluation map}:
\begin{displaymath}
\ev:\mathbf Z[L_n^*]\to \mathbf Z[x_1,\dotsc, x_n]
\end{displaymath}
from the monoid algebra of $L_n^*$ onto the ring of polynomials in $n$ variables.
In the algebra $\mathbf Z[L_n^*]$, define elements
\begin{align*}
  \mathbf H_k & = \sum_{\text{$w\in L_n^*$ is a row of length $k$}} w\\
  \mathbf E_k & = \sum_{\text{$w\in L_n*$ is a column of length $k$}} w
\end{align*}
for every positive integer $k$.
Then Exercise~\ref{exercise:row-col} can be restated as the identities:
\begin{displaymath}
  e_k = \ev(\mathbf E_k) \text{ and } h_k = \ev(\mathbf H_k).
\end{displaymath}
The evaluation map has a large kernel; is domain is the free algebra, and it maps onto the polynomial algebra.
Its image contains our primary object of interest---the algebra of symmetric polynomials in $n$ variables.
In the next few sections, we shall learn about an equivalence relation ``$\equiv$'' on $L_n^*$, called \emph{Knuth equivalence}, such that the resulting quotient monoid $\pl(L_n):=L_n^*/\equiv$ (called the \emph{plactic monoid}) has the property that the subalgebra of $\mathbf Z[\pl(L_n)]$ generated by the elements $\{\mathbf E_k\}_{k=1}^\infty$ or the elements $\{\mathbf H_k\}_{k=1}^\infty$ is isomorphic to the subalgebra of symmetric polynomials in $\mathbf Z[x_1,\dotsc, x_n]$ under the evaluation map.
\subsection{The Plactic Monoid}
\label{sec:plactic-monoid}
The plactic monoid $\pl(L_n)$ is the quotient of $L^*_n$ by the equivalence relation generated by the Knuth relations:
\begin{gather}
  \tag{$K1$}\label{eq:k1}
  xzy \equiv zxy \text{ if } x\leq y < z,
  \\
  \tag{$K2$}\label{eq:k2}
  yxz \equiv yzx \text{ if } x < y \leq z.
\end{gather}
Two words are said to be in the same plactic class if each can be obtained from the other by a sequence of moves of the form (\ref{eq:k1}) and (\ref{eq:k2}).
Since both sides of the Knuth relations have the same evaluation, it follows that the evaluation map $\ev:\mathbf Z[L^*_n]\to \mathbf Z[x_1,\dotsc,x_n]$ factors through the plactic monoid algebra $\mathbf Z[\pl(L_n)]$.
Let $E_k$ denote the image of $\mathbf E_k$ and $H_k$ denote the image of $\mathbf H_k$ in $\mathbf Z[\pl(L_n)]$.
\begin{exercise}
  Take $n=2$.
  Show that $E_1$ and $E_2$ commute in $\mathbf Z[\pl(L_2)]$.
  Show that they commute in $\mathbf Z[\pl(L_3)]$.
\end{exercise}
\begin{exercise}
  \label{exercise:forgotten}
  Define Sch\"utzenberger's forgotten relations by:
  \begin{gather}
    \tag{$F1$}\label{eq:f1}
    xzy \cong yxz \text{ if } x < y < z,
    \\
    \tag{$F2$}\label{eq:f2}
    zxy \cong yzx \text{ if } x \leq y \leq z.
  \end{gather}
  Let $F(L_n)$ denote the monoid $L_n^*/\cong$.
  Show that the images of $\mathbf E_1$ and $\mathbf E_2$ commute in $F(L_3)$.
\end{exercise}
\begin{exercise}
  Show that any evaluation-preserving equivalence on $L_3^*$ under which the images of $\mathbf E_1$ and $\mathbf E_2$ commute must include either the Knuth equivalences or Sch\"utzenberger's forgotten equivalences.
\end{exercise}
\begin{exercise}
  \begin{enumerate}
  \item Show that, if $\iota(r,x)=(y,r')$ (as in Section~\ref{sec:schenst-insert-algor}), then $rx\equiv yr'$.
  \item Show that, if $tx\equiv \ins(t,x)$. Here tableaux are to be identified with their reading words.
  \item Show that, for every $w\in L_n^*$, $w\equiv P(w)$.
  \end{enumerate}
\end{exercise}
\subsection{The plactic Pieri rules}
Observe that, if for any tableau $t$ and $x\in L_n$, if $t'=\ins(t,x)$, then $\shape(t')$ is obtained by adding one box to $\shape(t)$.
\begin{lemma}
  \label{lemma:two-step}
  Let $t$ be the (the reading word of) a semistandard tableau in $L_n^*$ and $x,y$ be letters in $L_n$.
  Let $t'=\ins(t,x)$ and $t''=\ins(t',y)$.
  Let $a$ be the box added to $\shape(t)$ to obtain $\shape(t')$, and $b$ be the box added to $\shape(t')$ to obtain $\shape(t'')$.
  If $x\leq y$, then $b$ lies in a column strictly to the right of the column of $a$.
  If $x>y$, then $b$ lies in a row strictly below the row of $a$.
\end{lemma}
\begin{proof}
  Let $a_1\dotsb a_k$ be the first row of $t$.

  Suppose $x\leq y$.
  Consider first the case where $a_k\leq y$.
  If $a_k\leq x\leq y$, then the result is obvious.
  If $x<a_k$ and $y\geq a_k$, then $b$ lies in the $(k+1)$st column, whereas $x$ bumps some letter $x'\leq y$ to a lower row.
  This letter cannot come to rest in the $(k+1)$st column because that would violate the fact that columns increase strictly in a semistandard tableau.

  Now consider the case where $a_k>y$.
  Then $x$ bumps out $x'$ and $y$ bumps out $y'$ with $x'\leq y'$.
  The problem is now reduced to the tableau obtained by removing the top row of $t$, allowing for the application of induction.
  In the base case (where the original tableau $t$ is a row), $a$ and $b$ are the first and second boxes in the second row of $t''$.

  Now suppose $x>y$.
  If $x\geq a_k$, then $x$ first comes to rest at the end of the first row in $t'$, but then $y$ bumps some element of the first row of $t'$ up to a lower row in $t''$.
  So $a$ lies in the first row and $b$ in a lower row.
  If $x\leq a_k$, then the elements $x'$ and $y'$ bumped out from the first row by $x$ and $y$ respectively again satisfy $x'>y'$, allowing for an inductive argument.
\end{proof}
\begin{theorem}
  [Plactic Pieri rules]
  \label{theorem:plactic-pieri}
  Let $\Tab_n(\lambda)$ denote the set of all semistandard tableaux of shape $\lambda$ and entries in $L_n$.
  Let $R_k(L_n)$ denote the set of rows of length $k$ in $L_n^*$.
  Then the map $(t,r)\mapsto P(tr)$ defines a bijection:
  \begin{displaymath}
    \Tab_n(\lambda)\times R_k(L_n) \to \coprod_{\mu/\lambda \text{ is a horizontal strip of size $k$}} \Tab_n(\mu).
  \end{displaymath}
  Let $C_k(L_n)$ denote the set of columns of length $k$ in $L_n^*$.
  Then the map $(t,c)\mapsto P(tc)$ defines a bijection:
  \begin{displaymath}
    \Tab_n(\lambda)\times C_k(L_n) \to \coprod_{\mu/\lambda \text{ is a vertical strip of size $k$}} \Tab_n(\mu).
  \end{displaymath}
\end{theorem}
\begin{proof}
  Lemma~\ref{lemma:two-step} implies that $\shape(P(wr))$ is obtained from\linebreak $\shape(P(w))$ by adding a horizontal strip, and that $\shape(P(wc))$ is obtained from $\shape(P(w))$ by adding a vertical strip.
  The bijectivity can be shown by repeated used of the $\del$ algorithm (because we know which box from $\mu$ has to be removed at each step).
\end{proof}
\begin{exercise}
  \label{exercise:plactic-pieri}
  Define an element of $\mathbf Z[L_n^*]$ by $\mathbf S_\lambda=\sum_{t\in \Tab_n(\lambda)} t$.
  Let $S_\lambda$ denote the image of $\mathbf S_\lambda$ in $\mathbf Z[\pl(L_n)]$.
  Use Theorem~\ref{theorem:plactic-pieri} to show that:
  \begin{align*}
    S_\lambda H_k&=\sum_{\mu/\lambda \text{ is a horizontal strip of size $k$}} S_\mu,\\
    S_\lambda E_k&=\sum_{\mu/\lambda \text{ is a vertical strip of size $k$}} S_\mu,\\
  \end{align*}
\end{exercise}
\begin{exercise}
  Show that $E_k$ and $E_l$ (also $H_k$ and $H_l$) commute in $\mathbf Z[\pl(L_n)]$ for all positive integers $n$, $k$ and $l$.
\end{exercise}
\begin{corollary}
  [Kostka's definition of Schur polynomials]
  \label{corollary:kostka-def-schur}
  For every partition $\lambda$ and every positive integer $n$, we have:
  \begin{displaymath}
    s_\lambda(x_1,\dotsc,x_n) = \sum_{t\in\Tab_n(\lambda)} x^t.
  \end{displaymath}
\end{corollary}
\begin{proof}
  Note that $\ev(\mathbf S_\lambda) = \sum_{t\in \Tab_n(\lambda)} x^t$.
  Then $\ev(\mathbf S_{(n)}) = h_n$, and $\ev(\mathbf S_{(1^n)}) = e_n$, just like the Schur polynomials.
  Moreover, since the evaluation map $\mathbf Z[L_n^*]$ factors through $\mathbf Z[\pl(L_n^*)]$, Exercise~\ref{exercise:plactic-pieri} implies that the polynomials $\ev(\mathbf S_\lambda)$ satisfy the Pieri rule, just like the Schur polynomials.
  This suffices for the identities of Theorem~\ref{theorem:schur-to-eh} to hold with $s_\lambda$ replaced by $\ev(\mathbf S_\lambda)$.
  By the triangularity properties of Kostka numbers (Section~\ref{sec:triang-kostka-numb}), these identities uniquely determine the Schur polynomials, therefore $S_\lambda=s_\lambda$ for every partition $\lambda$.
\end{proof}
\subsection{The Lindstr\"om-Gessel-Viennot Lemma}
\label{sec:lgv}
Let $R$ be a commutative ring.
Let $S$ be any set of points, and $v:S\times S\to R$ be any function (we think of $w$ as a \emph{weight function}).
Given $s, t\in S$, a path in $S$ from $s$ to $t$ is is a sequence $\omega=(s=s_0,s_1,\dotsc,s_k=t)$ of distinct points in $S$.
We denote this by $\omega:s\to t$.
The weight of the path $\omega$ is defined to be:
\begin{displaymath}
  v(\omega) = v(s_0,s_1)v(s_1,s_2)\dotsb v(s_{k-1}, s_k).
\end{displaymath}
\begin{definition}
  [Crossing paths]
  Two paths $\omega=(s_0,\dotsc, s_k)$ and $\eta=(t_0,\dotsc,t_l)$ are said to cross if $s_i= t_j$ for some $0\leq i \leq k$ and $0\leq j \leq l$.
\end{definition}
\begin{definition}
  [Crossing condition]
  \label{definition:crossing-condition}
  Given a set $S$ of points, a weight function $v:S\times S\to R$, and a points $A_1,\dotsc,A_n$, and $B_1,\dotsc, B_n$, we say that the \emph{crossing condition} is satisfied if, whenever $1\leq i<j\leq n$ and $1\leq i'<j'\leq n$, and $\omega:i \to j'$ and $\eta: j\to i'$ are paths such that $v(\omega)\neq 0$ and $v(\eta)\neq 0$, then the paths $\omega$ and $\eta$ cross.
\end{definition}
Fix points $A_1,\dotsc, A_n$ and $B_1,\dotsc, B_n$ in $S$, and define an $n\times n$ matrix $(a_{ij})$ by:
\begin{displaymath}
  a_{ij} = \sum_{\omega:A_i\to B_j} v(\omega).
\end{displaymath}
\begin{theorem}
  [Lindstr\"om-Gessel-Viennot Lemma]
  \label{lemma:lgv}
  Assume that the crossing condition (Definition~\ref{definition:crossing-condition}) holds.
  Then the determinant of the matrix $(a_{ij})$ defined above is given by:
  \begin{equation}
    \label{eq:lgv}
    \det(a_{ij}) = \sum_{\omega_i:A_i\to B_i} v(\omega_1)\dotsb v(\omega_n),
  \end{equation}
  where the sum is over all $n$-tuples $(\omega_1,\dotsc, \omega_n)$ of pairwise non-crossing paths $\omega_i:A_i\to B_i$.
\end{theorem}
\begin{proof}
  Let $P$ be the set of all $n$-tuples of paths of the form:
  \begin{equation}
    \label{eq:paths}
    \bar\omega = (\omega_i:A_i\to B_{w(i)}, i=1,\dotsc, n),
  \end{equation}
  where $w$ is a permutation of $\{1,\dotsc, n\}$.
  Define the weight of $\bar \omega\in P$ by:
  \begin{displaymath}
    v(\bar\omega) = \prod_{i=1}^n v(\omega_i)
  \end{displaymath}
  and its sign by $\epsilon(\bar\omega)=\epsilon(w)$.
  Then the determinant on the left hand side of (\ref{eq:lgv}) expands to the sum:
  \begin{equation}
    \label{eq:det-expansion}
    \sum_{\bar\omega\in P} \epsilon(\bar\omega)v(\bar\omega).
  \end{equation}
  The cancelling involution $I:P\to P$ is defined by \emph{swapping the first crossing of the first path that crosses another path}: given $\bar\omega$ as in (\ref{eq:paths}), if the paths are pairwise non-crossing, then $\bar\omega$ is a fixed point for $I$.
  In this case the crossing condition implies that $w$ is the identity permutation.
  Otherwise, take the least $i$ such that the path $\omega_i=(s_0,\dotsc,s_k)$ crosses another path, and then the least $j$ such that $\omega_j=(t_0,\dotsc,t_l)$ crosses $\omega_i$.
  Let $m$ be the smallest number such that a point $s_m$ of $\omega_i$ lies in the path $\omega_j$, say $s_m=t_r$.
  Let $I(\bar\omega)$ be the family of paths obtained from $\bar\omega$ by modifying $\omega_i$ and $\omega_j$ to $\omega'_i$ and $\omega'_j$ as follows:
  \begin{align*}
    \omega'_i & = (s_0,\dotsc, s_m, t_{r+1}, \dotsc, t_l),\\
    \omega'_j & = (t_0,\dotsc, t_r, s_{m+1}, \dotsc, s_k).
  \end{align*}
  Clearly, $v(I(\bar\omega)) = v(\bar\omega)$ and $\epsilon(I(\bar\omega)) = - \epsilon(\bar\omega)$.
  It is not hard to see that $I$ is an involution.
  This involution cancels out all the terms in (\ref{eq:det-expansion}) except those that occur on the right hand side of (\ref{eq:lgv}).
\end{proof}
\subsection{The Jacobi-Trudi Identities}
\label{sec:jacobi-trudi-ident}
We have seen that the Kostka numbers can be used to express complete and elementary symmetric polynomials in terms of Schur polynomials.
The reverse operation---that of expressing Schur polynomials in terms of complete or elementary symmetric polynomials---is done by the Jacobi-Trudi identities:
\begin{theorem}
  [Jacobi-Trudi identities]
  For every integer partition $\lambda=(\lambda_1,\dotsc,\lambda_l)$ form the $l\times l$ matrices with $(i,j)$th entry $h_{\lambda_i-i+j}$ and $e_{\lambda'_i-i+j}$ respectively.
  Then
  \begin{displaymath}
    s_\lambda = \det(h_{\lambda_i-i+j})= \det(e_{\lambda'_j+i-j}).
  \end{displaymath}
\end{theorem}
\begin{proof}
  The Jacobi-Trudi identities can be proved using the Lindstr\"om-Gessel-Viennot lemma (Theorem~\ref{lemma:lgv}).
  For the first identity take $S$ to be the positive cone in the the integer lattice:
  \begin{displaymath}
    S = \{(i,j)\mid i\geq 0,\; j>0 \text{ are integers}\}.
  \end{displaymath}
  Set the weight $v((i,j), (i+1,j))$ of each rightward horizontal edge to be $x_j$ for $j=1,\dotsc, n$, the weight of each upward vertical edge $v((i,j), (i,j+1))$ to be $1$ for all $j=1,\dotsc,n-1$.
  The remaining weights are all zero.
  \begin{lemma}
    \label{lemma:entry}
    For all integers $i>0$ and $k\geq 0$, we have:
    \begin{displaymath}
      \sum_{\omega:(i, 1)\to (i+k, n)} v(\omega) = h_k(x_1,\dotsc,x_n).
    \end{displaymath}
  \end{lemma}
  \begin{proof}
    Only rightward or upward steps have non-zero weights.
    So every path with non-zero weight is composed of unit upward and rightward steps.
    A path with non-zero weight from $(i,1)$ to $(i+k, n)$ must have exactly $k$ rightward steps, say in rows $1\leq j_1\leq j_2 \dotsb \leq j_k\leq n$.
    The weight of such a path is $x_{j_1}\dotsb x_{j_k}$, and hence, the sum of the weights of all such paths is $h_k(x_1,\dotsc, x_n)$. For an example, see Fig.~\ref{fig:path_example}.
  \end{proof}
  \begin{figure}
    \centering
    \includegraphics[width=0.7\textwidth]{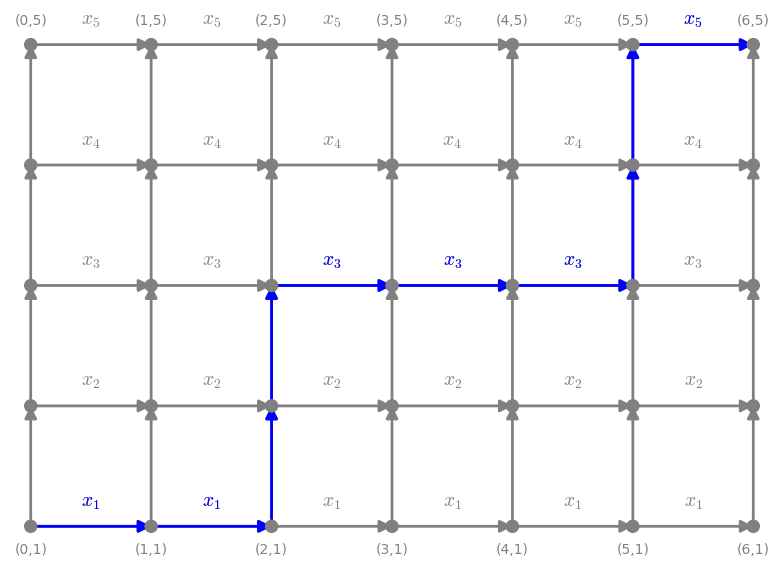}
    \caption{A path from $(0, 1)$ to $(6, 5)$ whose weight is the monomial $x_1^2x_3^3x_5$ in $h_6(x_1,\dotsc,x_5)$.}
    \label{fig:path_example}
  \end{figure}
  Given $\lambda=(\lambda_1,\dotsc,\lambda_l)$, and working with $n$ variables $x_1,\dotsc,x_n$, let $A_i = (l-i, 1)$ and $B_i=(\lambda_i+l-i, n)$ for $i=1,\dotsc, l$.
  Then  by Lemma~\ref{lemma:entry},
  \begin{displaymath}
    \sum_{\omega:A_i\to B_j} v(\omega) = h_{\lambda_j+i-j}.
  \end{displaymath}
  \begin{figure}
    \centering
    \raisebox{-0.5\height}{\includegraphics[width=0.6\textwidth]{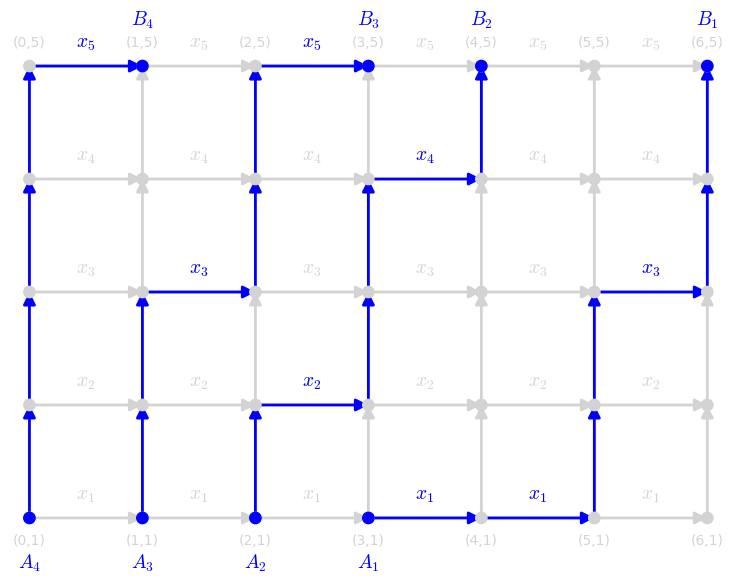}}\hspace{1cm}\raisebox{-0.5\height}{\ytableausetup{nosmalltableaux}$\ytableaushort{113,24,35,5}$}
    \caption{Non-crossing paths and corresponding tableau.}
    \label{fig:jacobi-trudi}
  \end{figure}
  So the left-hand-side of the first Jacobi-Trudi identity is the left-hand-side of the Lindstr\"om-Gessel-Viennot lemma.
  The right hand side of the Lindstr\"om-Gessel-Viennot lemma consists of a sequence of non-crossing paths $(\omega_1,\dotsc,\omega_n)$, where $\omega_i:A_i\to B_i$.
  Reading the row numbers of the horizontal steps in $\omega_i$ gives a weakly increasing sequence of integers $1\leq k_1 \leq \dotsb \leq k_{\lambda_i}\leq n$.
  Enter these numbers into the $i$th row of the Young diagram of $\lambda$ for $i=1,\dotsc,n$.
  Since the paths are non-crossing, the $j$th rightward step of $\omega_i$ must be strictly higher than the $j$th rightward step of $\omega_{i+1}$.
  This means that the columns of the resulting numbering are strictly increasing, resulting in a semistandard tableau of shape $\lambda$ (for an example, see Figure~\ref{fig:jacobi-trudi}).
  Thus, it follows from the Lindstr\"om-Gessel-Viennot lemma that
  \begin{displaymath}
    \det(h_{\lambda_j+i-j}) = \sum_{t\in \tab(\lambda)} x^t.
  \end{displaymath}

  For the second Jacobi-Trudi identity take
  \begin{displaymath}
    S = \{(i,j)\mid i\geq 0,\;j\geq 0\}.
  \end{displaymath}
  Define the weight of each upward vertical edge $v((i,j),(i,j+1))$ to be $1$ (as before) and the weight of a diagonal edge in the upper-right direction $v((i,j-1),(i+1,j))$ to be $x_j$; all other weights are zero.
  For the new weights, the analog of Lemma~\ref{lemma:entry} is:
  \begin{lemma}
    \label{lemma:entry-e}
    For all integers $i>0$ and $k>0$, we have:
    \begin{displaymath}
      \sum_{\omega:(i,0)\to (i+k,n)} v(\omega) = e_k(x_1,\dotsc,x_n).
    \end{displaymath}
  \end{lemma}
  \begin{proof}
    Every path with non-zero weights consists of unit upward or upper-rightward diagonal steps.
    A path with non-zero weight from $(i,0)$ to $(i+k,n)$ must have $n$ such steps, of which $k$ must be diagonal.
    If the steps ending in rows $1\leq j_1<\dotsc<j_k\leq n$ are the diagonal steps, then the path has weight $x_{j_1}\dotsb x_{j_k}$.
    For an example of such a path, see Fig.~\ref{fig:example_path_e}.
    Summing over all possible paths gives $e_k(x_1,\dotsc,x_n)$.
  \end{proof}
  \begin{figure}
    \centering
    \includegraphics[width=0.7\textwidth]{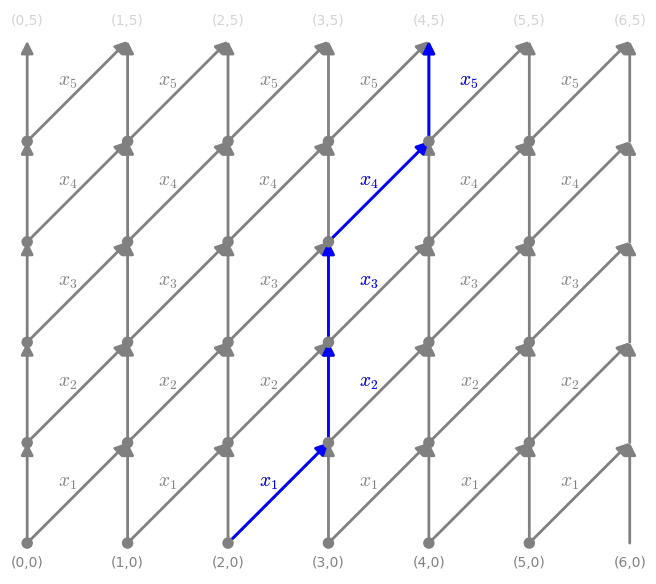}
    \caption{A path from $(2,0)$ to $(4,5)$ whose weight is the monomial $x_1x_4$ in $e_2(x_1,\dotsc,x_5)$.}
    \label{fig:example_path_e}
  \end{figure}
  Suppose that the conjugate partition of $\lambda$ is $\lambda'=(\lambda'_1,\dotsc,\lambda'_k)$.
  In order to apply the Lindstr\"om-Gessel-Viennot lemma to obtain the second Jacobi-Trudi identity, take $A_i=(k-i, 0)$ and $B_i=(\lambda'_i+k-i, n)$ for $i=1,\dotsc,k$.
  Then by Lemma~\ref{lemma:entry-e},
  \begin{displaymath}
    \sum_{\omega_i:A_i\to B_j} v(\omega) = e_{\lambda'_j+i-j}.
  \end{displaymath}
  So the left-hand-side of the second Jacobi-Trudi identity is the left-hand-side of the Lindstr\"om-Gessel-Viennot lemma.

  The right hand side of the Lindstr\"om-Gessel-Viennot lemma consists of a sequence of non-crossing paths $(\omega_1,\dotsc,\omega_n)$, where $\omega_i:A_i\to B_i$.
  Reading the row numbers where the upper-rightward steps in $\omega_i$ terminate gives a strictly increasing sequence of integers $1\leq j_1 < \dotsb < j_{\lambda'_i}\leq n$.
  Enter these numbers into the $i$th column of the Young diagram of $\lambda$.
  Since the paths are non-crossing, the $j$th upper-rightward step of $\omega_i$ must be no lower than the $j$th upper-rightward step of $\omega_{i+1}$.
  This means that the rows of the resulting numbering are weakly increasing, resulting in a semistandard tableau of shape $\lambda$
  \begin{figure}
    \centering
    \raisebox{-0.5\height}{\includegraphics[width=0.7\textwidth]{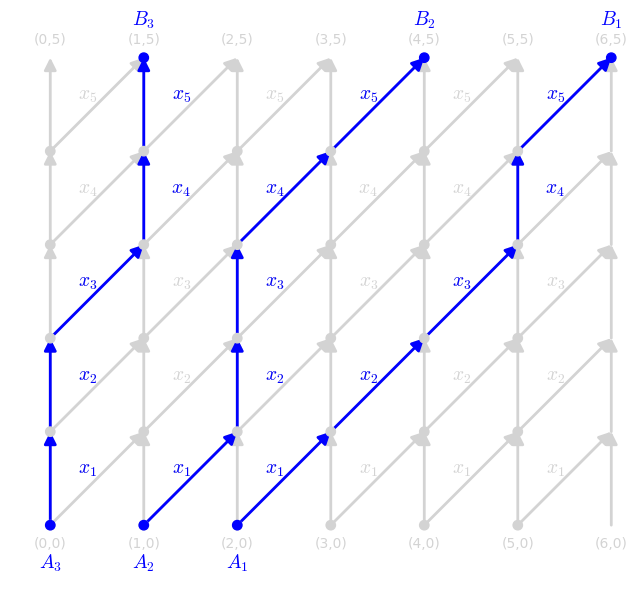}}\hspace{1cm}\raisebox{-0.5\height}{$\ytableaushort{113,24,35,5}$}
    \caption{Non-crossing paths and corresponding tableau.}
    \label{fig:jte}
  \end{figure}
  (for an example, see Figure~\ref{fig:jte}).
  Thus, it follows from the Lindstr\"om-Gessel-Viennot lemma that
  \begin{displaymath}
    \det(e_{\lambda_j+i-j}) = \sum_{t\in \tab(\lambda')} x^t,
  \end{displaymath}
  proving the second Jacobi-Trudi identity.
\end{proof}
\subsection{Skew-Schur Polynomials}
\label{sec:skew-schur-functions}
In the proof of the first Jacobi-Trudi identity for $s_\lambda$, where $\lambda=(\lambda_1,\dotsc, \lambda_l)$, we used:
\begin{align*}
  A_i & = (l-i, 1),\\
  B_i & = (\lambda_i+l-i, n)
\end{align*}
for $i=1,\dotsc, l$.
Now suppose $\mu = (\mu_1,\dotsc, \mu_l)$ is a partition (possibly padded with zero's so that it has the same number of parts as $\lambda$) such that $\mu\subset \lambda$, then we can take:
\begin{align*}
  A_i & = (\mu_i+l-i, 1),\\
  B_i & = (\lambda_i+l-i, n).
\end{align*}
Consider a collection of non-crossing paths $\omega_i:A_i\to B_i$, $i=1,\dotsc,l$.
The path $\omega_i$ has $\lambda_i-\mu_i$ horizontal steps.
If these steps occur in rows $1\leq k_1\leq \dotsb \leq k_{\lambda_i-\mu_i}\leq n$, then enter the integers $k_1,\dotsc,k_{\lambda_i-\mu_i}$ into the $i$th row of the skew-shape.
As in the proof of the first Jacobi-Trudi identity, this results in a semistandard tableau of skew-shape $\lambda/\mu$.
\begin{figure}[h]
  \centering 
  \raisebox{-0.5\height}{\includegraphics[width=0.6\textwidth]{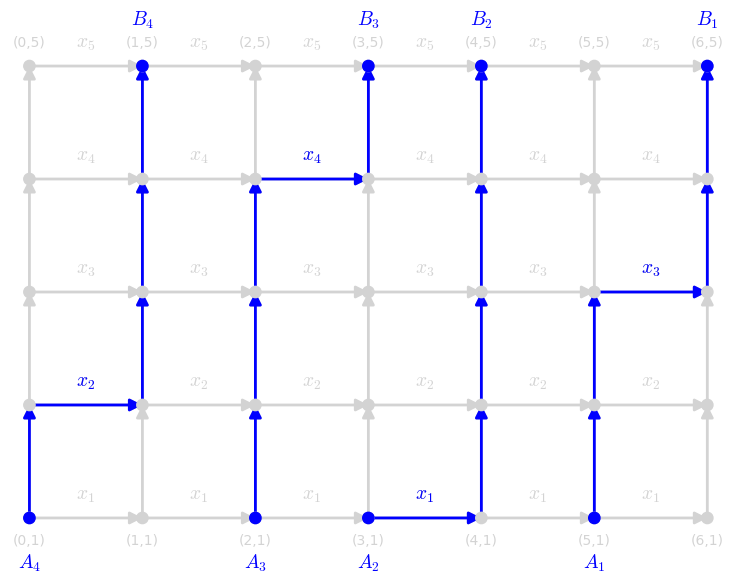}}  \hspace{1cm}  \raisebox{-0.5\height}{$\ytableaushort{\none\none 3,\none 1,\none 4,2}$}
  \caption{Non-crossing paths and the corresponding skew-tableau.}
  \label{fig:skew-jth}
\end{figure}
For an example, see Figure~\ref{fig:skew-jth}.

Therefore the symmetric polynomial
\begin{equation}
  \label{eq:skew-jth}
  s_{\lambda/\mu}(x_1,\dotsc,x_n) = \det(h_{\lambda_j-\mu_j+i-j}),
\end{equation}
by the Lindstr\"om-Gessel-Viennot, is also given by
\begin{equation}
  \label{eq:skew-kostka}
  s_{\lambda/\mu}(x_1,\dotsc,x_n) = \sum_{t\in \Tab_n(\lambda/\mu)} x^t.
\end{equation}
The polynomials $s_{\lambda/\mu}$ generalize the Schur polynomials and are called \emph{skew-Schur polynomials}.
A modification of the proof of the second Jacobi-Trudi identity gives:
\begin{equation}
  \label{eq:skew-jti}
  s_{\lambda/\mu}(x_1,\dotsc,x_n) = \det(e_{\lambda'_j-\mu'_j+i-j}),
\end{equation}
\begin{exercise}
  Expand the skew-Schur polynomial $s_{(2,1)/(1)}(x_1,x_2,x_3)$ in the basis of Schur polynomials.
\end{exercise}
\subsection{Giambelli's Identity}
\label{sec:giambelli}
\begin{definition}
  [Frobenius coordinates]
  Let $\lambda=(\lambda_1,\dotsc,\lambda_l)$ be a partition.
  Its \emph{Durfee rank} $d$ is defined to be the largest integer $i$ such that $(i,i)$ lies in the Young diagram of $\lambda$.
  Let $\alpha_i$ denote the number of cells in the $i$th row that lie strictly to the right of $(i,i)$ in the Young diagram of $\lambda$.
  Similarly let $\beta_i$ denote the number of cells in the $i$th column that lie strictly below $(i,i)$.
  Clearly $\alpha_1>\dotsb>\alpha_d$, $\beta_1>\dotsb>\beta_d$, and the Young diagram of $\lambda$ can be recovered from the data $(\alpha|\beta)=(\alpha_1,\dotsc,\alpha_d|\beta_1,\dotsc,\beta_d)$, which are called the \emph{Frobenius coordinates} of $\lambda$\footnote{While constructing the character tables of symmetric groups, Frobenius used these coordinates to index the irreducible representation, while he used the ordinary coordinates to index the conjugacy classes.}.
\end{definition}
\begin{example}
  The hook partition $(a+1,1^b)$ has Frobenius coordinates $(a|b)$.
  Hook partitions are precisely those partitions which have Durfee rank $1$.
  The partition with Frobenius coordinates $(5,2,1|4,3,0)$ is $(6,4,4,2,2)$.
  If $\lambda$ has Frobenius coordinates $(\alpha|\beta)$, then its conjugate $\lambda'$ has Frobenius coordinates $(\beta|\alpha)$.
  The size of a partition with Durfee rank $d$ and Frobenius coordinates $(\alpha|\beta)$ is $d+|\alpha|+|\beta|$. 
\end{example}
Schur polynomials of hook partitions can be calculated using Exercise~\ref{exercise:hook-schur}, which, when written is terms of Frobenius coordinates, becomes:
\begin{equation}
  \label{eq:hook-schur-frob}
  s_{(a|b)} = \sum_{l=0}^b (-1)^l h_{a+l+1}e_{b-l}
\end{equation}
\begin{theorem}
  [Giambelli's formula]
  For a partition $(\alpha_1,\dotsc,\alpha_d|\beta_1,\dotsc,\beta_d)$  in Frobenius coordinates,
  \begin{equation}
    \label{eq:giambelli}
    s_{(\alpha|\beta)} = \det(s_{(\alpha_j|\beta_i)})_{d\times d}.
  \end{equation}
  Note that the determinant on the right consists of hook-partition Schur polynomials, which are given by \textup{(\ref{eq:hook-schur-frob})}.
\end{theorem}
\begin{example}
  The Schur polynomial for $\lambda=(4,4,3,1)=(3,2,0|3,1,0)$ can be computed as:
  \begin{displaymath}
    s_{(3,2,1|3,1,0)} = \det
    \begin{pmatrix}
      s_{(3|3)} & s_{(2|3)} & s_{(0|3)}\\
      s_{(3|1)} & s_{(2|1)} & s_{(0|1)}\\
      s_{(3|0)} & s_{(2|0)} & s_{(0|0)}\\
    \end{pmatrix}
  \end{displaymath}
\end{example}
\begin{proof}
  Giambelli's identity can be proved using the Lindstr\"om-Gessel-Viennot Lemma.
  Let $\lambda=(\alpha_1,\dotsc,\alpha_d|\beta_1,\dotsc,\beta_d)$ be given.
  Working with $n$ variables, set
  \begin{displaymath}
    S = \{(i,j)\mid 1\leq j \leq n\; i\geq 0\} \cup \{(-i,j)\mid 1\leq j\leq n+1\; i>0\}.
  \end{displaymath}
  \begin{figure}[h]
    \centering
    \raisebox{-0.5\height}{\includegraphics[width=0.7\textwidth]{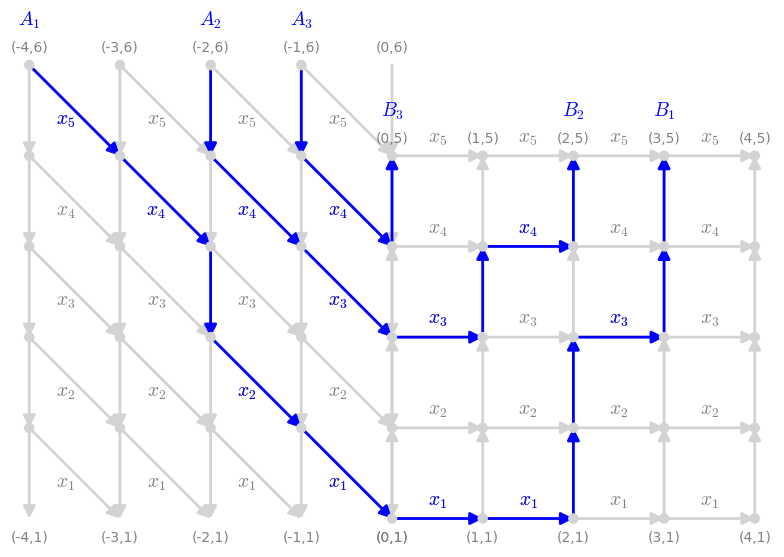}}\hspace{1cm}$\ytableaushort{1113,2334,444,5}$
    \caption{Non-crossing paths and corresponding tableau.}
    \label{fig:giambelli}
  \end{figure}
  \ytableausetup{smalltableaux}
  Define a weight function as follows (see Fig.~\ref{fig:giambelli}):
  \begin{align*}
    v((-i,j+1), v(-i,j)) & = 1 \text{ for all } i> 0, 0 \leq j \leq n,\\
    v((-i,j+1), v(-(i-1),j)) & = x_j \text{ for all } i>0, j\geq 1,\\
    v((i,j), (i+1,j)) & = x_j \text{ for all } i\geq 0, 1\leq j\leq n,\\
    v((i,j),(i,j+1)) & = 1 \text{ for all } i\geq 0, 1\leq j \leq n.
  \end{align*}
  All other weights are set to zero.
  A path from $(-(b-1),n+1)$ to $(a,n)$ has weight $x^t$, where $t$ is a semistandard tableau of shape $(a|b)$ in $x_1,\dotsc,x_n$.
  For instance, the path in Fig.~\ref{fig:giambelli} from $(-4,6)$ to $(3,5)$ corresponds to the tableau $\ytableaushort{1113,2,4,5}$.

  Thus, setting $A_i=(-(\beta_i-1),n+1)$ and $B_i=(\alpha_i,n)$ for $i=1,\dotsc,d$, 
  the $(i,j)$th entry of the determinant on the right hand side of (\ref{eq:giambelli}) can be written as:
  \begin{equation}
    \label{eq:giambelli-det}
    \sum_{\omega:A_i\to B_j} v(\omega).
  \end{equation}
  It is not hard to see that the non-crossing path configurations $\{\omega_i:A_i\to B_i\}$ correspond to semistandard tableau of shape $(\alpha|\beta)$ (for an example, see Fig.~\ref{fig:giambelli}).
\end{proof}
\subsection{Greene's Theorem}
\label{sec:greenes-theorem}
Given a word $w = (a_1,\dotsc,a_k)\in L_n^*$, a subword is a word of the form $w' = a_{i_1}\dotsb a_{i_r}$, where $1\leq i_1<\dotsb < i_r\leq k$.
This section is concerned with the enumeration of subwords which are rows or columns.
For the purposes of such enumeration, given $1 \leq j_1<\dotsb <j_r\leq k$, $w'' = a_{j_1}\dotsb a_{j_r}$ will be considered to be a different subword from $w'$ \emph{even if} $w'=w''$ as words, unless the indices $j_1,\dotsc, j_r$ coincide with the indices $i_1,\dotsc,i_r$.
Two subwords of $w$ will be said to be disjoint, if their indexing sets are disjoint.
\begin{example}
  The word $w=111$ has three subwords of length two, all of them equal to $11$.
  No two of these subwords are disjoint.
  However, each of them is disjoint from a subword of $w$ of length one.
\end{example}
\begin{definition}
  [Greene invariants]
  \label{definition:Greene-invars}
  Given $w \in L_n^*$, for each integer $k\geq 0$, let $l_k(w)$ denote the maximum cardinality of a union of $k$ pairwise disjoint weakly increasing subwords of $w$.
  Let $l'_k(w)$ denote the maximum cardinality of a union of $k$ pairwise disjoint strictly decreasing subwords of $w$.
\end{definition}
\begin{example}
  If $w=2133$, then $l_1(w)=3$, $l_k(w)=4$ for all $k\geq 2$.
  Also, $l'_1(w) = 2$, $l'_2(w)=3$, and $l'_k(w)=4$ for all $k\geq 3$.
\end{example}
\begin{theorem}
  [Greene's Theorem]
  \label{theorem:Greene}
  Given a word $w$, define a partition $\lambda=(\lambda_1,\lambda_2,\dotsc)$ by $\lambda_k=l_k(w)-l_{k-1}(w)$ for each $k\geq 1$.
  Then $\lambda$ is the shape of $P(w)$.
  Moreover, if $\mu'_k=l'_k(w)-l'_{k-1}(w)$ for each $k\geq 1$, then $\mu=(\mu_1,\mu_2,\dotsc)$ is the partition conjugate to $\lambda$.
\end{theorem}
\begin{proof}
  Greene's theorem follows by putting together two relatively simple observations---the first is that the Greene invariants $l_k(w)$ and $l'_k(w)$ remain unchanged when either of the Knuth relations (\ref{eq:k1}) and (\ref{eq:k2}) is applied to $w$.
  The second is that when $w$ is the reading word of a semistandard tableau, then Greene's theorem holds.

  To see the first, suppose that $w$ is of the form $u_1xzyu_2$, with $x\leq y < z$, and arbitrary $u_1,u_2\in L_n^*$.
  Applying a Knuth transformation of the form (\ref{eq:k1}), $w$ transforms to $w'=u_1zxyu_2$.
  Any weakly increasing subword of $w'$ is also a weakly increasing subword of $w$, so $l_k(w')\leq l_k(w)$.
  On the other hand, if $v$ is a weakly increasing subword of $w$ of the form $v_1xzv_2$ (where $v_i$ is a subword of $u_i$ for $i=1,2$) it will not necessarily remain a weakly increasing subword of $w'$.
  However, $v_1xyv_2$ \emph{is} a weakly increasing subword of $w$.
  If a collection of $k$ weakly increasing subwords of $w$ contains $v_1xzv_2$ and another weakly increasing subword $v_1'yv_2'$, replacing them by $v_1yzv'_2$ and  $v_1'xv_2$ gives a collection of $k$ weakly increasing subwords of $w'$ of the same cardinality.
  It follows that $l_k(w')\geq l_k(w)$ also holds.
  Thus the Knuth transformation (\ref{eq:k1}) preserves the Greene invariants $l_k(w)$.
  Similar arguments can be used to show that both (\ref{eq:k1}) and (\ref{eq:k2}) preserve all the Greene invariants $l_k(w)$ and $l'_k(w)$.

  Now suppose $w$ is the reading word of a tableau of shape $\lambda$.
  Then the top $k$ rows of $w$ form a union of $k$ pairwise disjoint weakly increasing subwords of total size $\lambda_1+\dotsb+\lambda_k$.
  Also, if $v$ is a weakly increasing subword of $w$, then the fact that the columns of $t$ are strictly increasing (and that the rows are read from \emph{bottom to top}) implies that $v$ cannot contain more than one element from each column of $w$.
  Therefore, any collection of $k$ pairwise disjoint weakly increasing subwords of $w$ can have at most $k$ entries in each column of $w$.
  Thus no union of $k$ pairwise disjoint weakly increasing subwords of $w$ can have cardinality more than $\lambda_1+\dotsb+\lambda_k$
  Therefore, $l_k(w)=\lambda_1+\dotsb+\lambda_k$.

  Similarly, the leftmost $k$ columns of $w$ (read bottom to top) form a union of $k$ pairwise disjoint strictly decreasing subwords of $w$, and any such union can only contain $k$ elements from each row.
  It follows that if $\lambda'$ is the partition conjugate to $\lambda$, then $l'_k(w)=\lambda'_1+\dotsb + \lambda'_k$.
\end{proof}
\subsection{The Robinson-Schensted-Knuth Correspondences}
\label{sec:rsk}
For an $m\times n$ matrix $A=(a_{ij})$, the column word $u_A$, row word $v_A$ and their duals $\cev u_A$ and $\cev v_A$ are defined as follows:
\begin{align*}
  u_A &= 1^{a_{11}}2^{a_{12}}\dotsb n^{a_{1n}} \; 1^{a_{21}} 2^{a_{22}} \dotsb n^{a_{2n}} \; \dotsb \; 1^{a_{m1}} 2^{a_{m2}} \dotsb n^{a_{mn}}\\
  v_A &= 1^{a_{11}}2^{a_{21}}\dotsb m^{a_{m1}} \; 1^{a_{12}} 2^{a_{22}} \dotsb m^{a_{m2}} \; \dotsb \; 1^{a_{1n}} 2^{a_{2n}} \dotsb m^{a_{mn}}\\
  \cev u_A &= n^{a_{1n}} \dotsb 2^{a_{12}}1^{a_{11}}\; n^{a_{2n}}\dotsb 2^{a_{22}}1^{a_{21}} \; n^{a_{mn}} \dotsb 2^{a_{m2}} 1^{a_{11}}\\
  \cev v _A &= m^{a_{m1}}\dotsb 2^{a_{21}}1^{a_{11}} \; m^{a_{m2}}\dotsb 2^{a_{22}} 1^{a_{12}} \; \dotsb \; m^{a_{mn}}\dotsb  2^{a_{2n}} \dotsb 1^{a_{1n}}
\end{align*}
\begin{definition}
  [Robinson-Schensted-Knuth Correspondences]
  Define functions from integer matrices onto pairs of semistandard tableaux by:
  \begin{align*}
    \rsk(A) & = (P(u_A), P(v_A)),\\
    \rsk^*(A) & = (P^*(u_A), P(\cev v_A)).
  \end{align*}
\end{definition}
\begin{lemma}
  \label{lemma:same_shape}
  For every integer matrix $A$ with non-negative entries, the tableaux $P(u_A)$ and $P(v_A)$ have the same shape.
  For every zero-one matrix $A$, the tableaux $P^*(u_A)$ and $P(\cev v_A)$ have the same shape.
\end{lemma}
\begin{proof}
  The proof is an application of Greene's theorem (Theorem~\ref{theorem:Greene}).
  Any weakly increasing subword of $u_A$ comes from reading the column number of a sequence of entries $(i_1,j_1),\dotsc,(i_r,j_r)$ with repetitions of up to $a_{i_1j_1},\dotsc, a_{i_rj_r}$ respectively, with $i_1\leq \dotsb \leq i_r$ and $j_1\leq \dotsc\leq j_r$.
  If $A$ is an $m\times n$ matrix, its entries are indexed by the rectangular lattice $P_{mn} = \{(i,j)\mid 1\leq i\leq m,\;1\leq j\leq n\}$ which may be regarded as a poset under $(i,j)\leq (i',j')$ if $i\leq i'$ and $j\leq j'$.
  It follows that
  \begin{displaymath}
    l_k(u_A) = \max_{C_k} \Big\{\sum_{(i,j)\in C_k} a_{ij}\Big\}
  \end{displaymath}
  where the maximum is over all subsets $C_k$ of $P_{mn}$ which can be written as a union of $k$ chains in the partially ordered set $P_{mn}$.
  This description of the shape of $P(u_A)$ is invariant under interchanging the rows and columns of $A$, and therefore also the shape of $P(v_A)$.

  For the dual RSK correspondence, if $A$ is a $0$-$1$ matrix, note that a strictly increasing subword of $u_A$ comes from reading the column number of a sequence of entries $(i_1,j_1),\dotsc,(i_r,j_r)$ with entries equal to $1$, and with $i_1\leq \dotsb \leq i_r$ and $j_1<\dotsc<j_r$.
  Define a new partial order $P_{mn}$ by $(i,j)<(i',j')$ if $i\leq i'$ and $j<j'$.
  It follows that
  \begin{displaymath}
    l_k^*(u_A) = \max_{C_k} \Big\{\sum_{(i,j)\in C_k} a_{ij}\Big\}
  \end{displaymath}
  where the maximum is over all subsets $C_k\subset P_{mn}$ which can be written as a union of $k$ chains in the new partial order.
  On the other hand, the row numbers in the sequence of entries $(i_1,j_1),\dotsc,(i_r,j_r)$ form a weakly increasing subword of $\cev v_A$ if and only if $i_1\leq \dotsb \leq i_r$, and since the entries must come from distinct rows (since $\cev v_A$ reads each row in reverse order and all entries are $0$ or $1$), so $j_1<\dotsb < j_r$.
  Thus $l_k(\cev v_A)=l_k^*(u_A)$, so $P^*(u_A)$ and $P(\cev v_A)$ have the same shape.
\end{proof}
\begin{theorem}
  [Knuth's theorem]
  \label{theorem:knuth}
  Let $\mathbf M_{\mu\nu}$ denote the set of integer matrices with non-negative entries, row sums $(\mu_1,\dotsc, \mu_m)$, column sums $(\nu_1,\dotsc,\nu_n)$.
  Let $\Tab(\lambda,\mu)$ denote the set of semistandard tableaux of shape $\lambda$ and type $\mu$.
  Then $\rsk$ gives rise to a bijection:
  \begin{equation}
    \label{eq:rsk}
    \mathbf M_{\mu\nu}\tilde\to \coprod_\lambda \Tab(\lambda,\nu)\times \Tab(\lambda,\mu).
  \end{equation}
  Similarly let $\mathbf N_{\mu\nu}$ denote the set of zero-one matrices with row sums $(\mu_1,\dotsc, \mu_m)$ and column sums $(\nu_1,\dotsc,\nu_n)$.
  Then $\rsk^*$ gives rise to a bijection:
  \begin{equation}
    \label{eq:drsk}
    \mathbf N_{\mu\nu}\tilde\to \coprod_\lambda \Tab^*(\lambda,\nu)\times \Tab(\lambda,\mu).
  \end{equation}
\end{theorem}
\begin{proof}
  We will show that $\rsk$ is a bijection:
  \begin{displaymath}
    \mathbf M_{m\times n} \tilde\to \coprod_\lambda \Tab_n(\lambda)\times \Tab_m(\lambda).
  \end{displaymath}
  For the definition it is clear that this bijection will map the left hand side of (\ref{eq:rsk}) onto its right hand side.

  Let $A'$ be the matrix consisting of the first $m-1$ rows of $A$.
  Let $r=1^{a_{m1}}2^{a_{m2}}\dotsb n^{a_{mn}}$ be the column word of the last row of $A$.

  By inducting on the number of rows of $A$ (the base case of one-row matrices is easy), we have bijections:
  \begin{equation}
    \label{eq:bij1}
    \mathbf M_{m\times n} \leftrightarrow \mathbf M_{(m-1)\times n}\times R(L_n) \leftrightarrow \coprod_\lambda \Tab_n(\lambda)\times \Tab_{m-1}(\lambda)\times R(L_n)
  \end{equation}
  given by
  \begin{displaymath}
    A \leftrightarrow (A',r) \leftrightarrow (P(u_{A'}), P(v_{A'}), r).
  \end{displaymath}
  Here $R(L_n)$ denotes the set of all rows (weakly increasing words) in $L_n^*$.

  Define a function
  \begin{equation}
    \label{eq:bij2}
    \Tab_n(\lambda)\times \Tab_{m-1}(\lambda)\times R(L_n) \to \coprod_\mu \Tab_n(\mu)\times \Tab_m(\mu)
  \end{equation}
  by
  \begin{displaymath}
    (t_1, t_2, r)\mapsto (P(t_1r), t_2^{\uparrow m}),
  \end{displaymath}
  where $t_2^{\uparrow m}$ is the unique tableau with the same shape as $P(t_1r)$ such that if all the boxes containing $m$ are removed from it, then $t_2$ is obtained.

  It turns out that the above function is invertible.
  Given tableaux $(t'_1, t'_2)\in \Tab_n(\mu)\times \Tab_m(\mu)$.
  Let $t_2$ be the tableau obtained from $t'_2$ by removing all the boxes containing $m$.
  Let $\lambda$ be the corresponding shape.
  Obviously $\mu/\lambda$ is a horizontal strip.
  Applying the inverse of the first bijection in Theorem~\ref{theorem:plactic-pieri} recovers $t_1$ and $r$.

  Combining the bijections (\ref{eq:bij1}) and (\ref{eq:bij2}) gives rise to the RSK correspondence, which is therefore also a bijection.

  The proof for the bijectivity of $\rsk^*$ is similar (although with a few twists) and is left as an interesting exercise to the reader.
\end{proof}
\begin{exercise}
  [The Burge Correspondence]
  \label{exercise:burge}
  Define
  \begin{displaymath}
    \bur(A) = (P^*(\cev u_A), P^*(\cev v_A)).
  \end{displaymath}
  Show that $\bur$ is a bijection
  \begin{displaymath}
    \mathbf M_{\mu\nu}\tilde \to \coprod_\lambda \Tab^*(\lambda,\nu)\times \Tab^*(\lambda,\mu).
  \end{displaymath}
\end{exercise}
\subsection{The Littlewood-Richardson Rule}
\label{sec:littlewood-richardson}
\begin{lemma}
  \label{lemma:section}
  Given a partition $\lambda$, fix any $t_\lambda\in \Tab_m(\lambda)$.
  Then
  \begin{displaymath}
    \sum_{\{A_{m\times n}\mid P(v_A) = t_\lambda\}} x^{u_A} = s_\lambda(x_1,\dotsc,x_n).
  \end{displaymath}
\end{lemma}
\begin{proof}
  If $P(v_A) = t_\lambda$, a tableau of shape $\lambda$, $P(u_A)$ is also a semistandard tableau of shape $\lambda$.
  Moreover, for every semistandard tableau $t\in \Tab_n(\lambda)$,  by Knuth's theorem (Theorem~\ref{theorem:knuth}), there exists a unique $m\times n$ integer matrix $A$ such that $\rsk(A) = (t, t_\lambda)$.
  In other words, among matrices with $P(v_A)=t_\lambda$, there exists a unique matrix such that $u_A\equiv t$.
  The lemma now follows from Kostka's definition of Schur polynomials (Corollary~\ref{corollary:kostka-def-schur}).
\end{proof}
\begin{theorem}
  [Littlewood-Richardson Rule]
  Let $\alpha$, $\beta$ and $\lambda$ be partitions.
  Let $t_\beta$ be any semistandard tableau in $\Tab_b(\beta)$ for some integer $b$.
  Let $c^\lambda_{\alpha\beta}$ denote the number of semistandard skew-tableaux of shape $\lambda/\alpha$ whose reading word is Knuth-equivalent to $t_\beta$, the unique tableau of shape and type $\beta$.
  Then
  \begin{displaymath}
    s_\alpha s_\beta = \sum_\lambda c^\lambda_{\alpha\beta} s_\lambda.
  \end{displaymath}
\end{theorem}
\begin{proof}
  Let $t_\alpha\in \Tab_a(\alpha)$ for some integer $a$.
  By Lemma~\ref{lemma:section}, we have:
  \begin{align}
    \nonumber s_\alpha s_\beta & = \sum_{\{A_{a\times n}\mid P(v_{A}) = t_\alpha\}} x^{u_{A}} \sum_{\{B_{b\times n}\mid P(v_B) = t_\beta\}} x^{u_{B}}\\
    \label{eq:sum-C} & = \sum_{C_{\alpha\beta}} x^{u_C},
  \end{align}
  where
  \begin{displaymath}
    C_{\alpha\beta} = \{\tbinom AB\mid P(v_A)=t_\alpha\text{ and } P(v_B)=t_\beta\}.
  \end{displaymath}
  Define a monoid homomorphism $\pi_a:L_{a+b}^*\to L^*_a$ by taking $\pi_a(w)$ to be the word obtained by discarding all letters in $w$ that are not in $\{1,\dotsc,a\}$.
  Also define $\pi_b:L_{a+b}^*\to L_b^*$ by taking $\pi_b(w)$ to be the word obtained discarding all letters in $w$ that are not in $\{a+1,\dotsc, a+b\}$, and then replacing the letter $a+i$ by $i$.
  It is not hard to see that $\pi_a$ and $\pi_b$ preserve the Knuth relations (\ref{eq:k1}) and (\ref{eq:k2}).
  Let $C=\tbinom AB$.
  The words $v_A$ and $v_B$ are obtained from $v_C$ as follows:
  \begin{displaymath}
    v_A = \pi_a(v_C) \text{ and } v_B = \pi_b(v_C).
  \end{displaymath}
  So $P(v_A) = P(\pi_a(v_C)) \equiv \pi_a(P(v_C))$.
  Since $\pi_A$ is a restriction to the first few letters of $L_{a+b}$, $\pi_a(P(v_C))$ is still a tableau.
  So $P(v_A) = \pi_a(P(v_C))$.
  Also, $P(v_B) \equiv \pi_b(P(v_c))$.
  Therefore
  \begin{align*}
    C_{\alpha\beta} & = \{C \mid \pi_a(P(v_C)) = t_\alpha \text{ and } \pi_b(P(v_C))\equiv  t_\beta\}\\
    & = \coprod_t \{C\mid P(v_C)=t\},
  \end{align*}
  the disjoint union being over
  \begin{displaymath}
    \{t\in \Tab(L_{a+b})\mid \pi_a(t)=t_\alpha \text{ and } \pi_b(t)\equiv t_\beta\},
  \end{displaymath}
  Here $\Tab(L_n)$ is used to denote the set of all semistandard tableaux in $L_n^*$.
  Let $t$ be a tableau $t$ in the above set of shape $\lambda$.
  The entries of this tableau in the cells of $\alpha$ are completely fixed by the condition $\pi_a(t)=t_\alpha$.
  Subtracting $a$ from the remaining entries gives rise to a skew-tableau of shape $\lambda/\alpha$ whose reading word is equivalent to $t_\beta$.
  The number of such tableaux is, by definition, $c^\lambda_{\alpha\beta}$. Therefore,
  \begin{align*}
    s_\alpha s_\beta &= \sum_\lambda \sum_{\{t\in \Tab_b(\lambda/\alpha)\mid t\equiv t_\beta\}} \sum_{\{C\mid P(v_C)=t\}} x^{u_C}\\
    &=\sum_\lambda c^\lambda_{\alpha\beta} s_\lambda \hspace{3cm}\text{ [using Lemma~\ref{lemma:section}]}
  \end{align*}
  as required.
\end{proof}
\begin{definition}
  [Yamanouchi Word]
  A word $w\in L_n^*$ is called a Yamanouchi word if $x^u$ is a monomial with weakly increasing powers for every suffix $u$ of $w$.
\end{definition}
\begin{lemma}
  \label{lemma:yamanouchi}
  A word $w$ is a Yamanouchi word of type $\lambda$ if and only if its plactic class contains the unique semistandard tableau $t^0_\lambda$ of shape $\lambda$ and type $\lambda$.
\end{lemma}
\begin{proof}
  Check that Knuth relations preserve \emph{Yamanouchiness}.
  The only Yamanouchi tableau of type $\lambda$ also has shape $\lambda$.
\end{proof}
In view of Lemma~\ref{lemma:yamanouchi}, the Littlewood-Richardson rule becomes:
\begin{theorem}
  For partition $\alpha,\beta,\lambda$, $c^\lambda_{\alpha\beta}$ is the number of semistandard tableaux of shape $\lambda/\alpha$ and type $\beta$ whose reading word is a Yamanouchi word.
\end{theorem}
\begin{exercise}
  If $c^\lambda_{\alpha\beta}>0$ then $\lambda\supset \alpha$ and $\lambda\supset \beta$.
\end{exercise}
\begin{exercise}
  Suppose $\lambda=(\lambda_1,\dotsc,\lambda_l)$, with $\lambda_1\leq m$.
  Let $\check\lambda=(m-\lambda_l,\dotsc,m-\lambda_1)$, and let $\Lambda$ denote the partition $(l,\dotsc,l)$ (with $m$ repetitions of $l$).
  Show that $c^\Lambda_{\lambda\mu}=1$.
\end{exercise}
\subsection{Skew-Schur polynomials and the Littlewood-Richardson Rule}
\label{sec:skew-lr}
Littlewood-Richardson coefficients also answer the question of expanding skew-Schur polynomials in terms of Schur polynomials:
\begin{theorem}
  [Expansion of skew-Schur polynomials]
  \label{theorem:skew-exp}
  Let $\lambda$, $\alpha$ and $\beta$ be partitions such that $\lambda\supset \alpha$.
  Then
  \begin{displaymath}
    s_{\lambda/\alpha}(x_1,\dotsc,x_n) = \sum_\beta c^\lambda_{\alpha\beta}s_\beta(x_1,\dotsc,x_n).
  \end{displaymath}
\end{theorem}
\begin{proof}
  We have:
  \begin{align*}
    s_{\lambda/\alpha}(x_1,\dotsc,x_n) & = \sum_{t\in \Tab_n(\lambda/\alpha)} x^t\\
      & = \sum_\beta \sum_{t_\beta\in \Tab_n(\beta)} \sum_{\{t\in \Tab_n(\lambda/\alpha)\mid P(t)=t_\beta\}} x^t\\
      & = \sum_\beta c^\lambda_{\alpha\beta} \sum_{t_\beta\in \Tab_n(\beta)} x^t\\
      & = \sum_\beta c^\lambda_{\alpha\beta} s_\beta,
  \end{align*}
  as required.
\end{proof}
\begin{exercise}
  For partitions $\alpha\subset \lambda$, show that
  \begin{displaymath}
    K_{\lambda/\alpha,\mu} = \sum_\beta c^\lambda_{\alpha\beta} K_{\beta\mu}.
  \end{displaymath}
  Here $K_{\lambda/\alpha}$ is the number of skew-tableaux of shape $\lambda/\alpha$ and type $\mu$.
\end{exercise}
\subsection{Sources}
\label{sec:notes-literature}
The arrangement of topics in these notes loosely follows the first Chapter of Manivel's book \cite{manivel}.
The use of labelled abaci to prove Schur polynomial identities is from Loehr \cite{loehr}.
A nice exposition of the LGV lemma can be found in Viennot's 2016 lectures at The Institute of Mathematical Sciences, Chennai \cite{imsc2016}.
The LGV-lemma proof of the Giambelli identity has been published by Stembridge \cite{stembridge}.
The treatment of the RSK correspondence and the Littlewood-Richardson rule is guided by Lascoux and Sch\"utzenberger\cite{ls}, and Lascoux, Leclerc and Thibon \cite{llt}.
However, I have not seen the precise definition of $\rsk$ and $\rsk^*$ that I used in these notes published elsewhere.
My use of integer matrices to prove the Littlewood-Richardson rule does not appear to be widespread.
Exercise~\ref{exercise:burge} describes a correspondence introduced by Burge in \cite{burge}.
Besides the above references, readers who wish to go deeper into the subject may consult Fulton's book on Young tableaux \cite{fulton}, Chapter~7 and Fomin's appendix in the second volume of Stanley's book on Enumerative Combinatorics \cite{ec2}.
The relationship between the RSK correspondence, its dual, symmetric functions, and representation theory is the subject of my own book on Representation Theory \cite{rtcv}, which is written at a relatively elementary level.
Finally, there is no single book which has had a greater impact on the theory of symmetric functions that Macdonald's classic \cite{macd}.

Many details of proofs and perspectives emerged in discussions with Digjoy Paul.
I benefited greatly from informal discussions with Bishal Deb, Sudhir Ghorpade, Sridhar P.~Narayanan, K.~N.~Raghavan, Vijay Ravikumar, Evgeny Smirnov, and  S.~Viswanath, and indeed all the participants of the ATM Workshop on Schubert Varieties, held at The Institute of Mathematical Sciences (see \url{https://www.atmschools.org/2017/atmw/sv}).
Darij Grinberg sent me a list of corrections to an earlier version of this article.
\bibliographystyle{abbrv}
\bibliography{refs}

\begin{thebibliography}{10}

\bibitem{burge}
W.~H. Burge.
\newblock Four correspondences between graphs and generalized young tableaux.
\newblock {\em Journal of Combinatorial Theory (A)}, 17, 1972.

\bibitem{fulton}
W.~Fulton.
\newblock {\em Young Tableaux: With Applications to Representation Theory and
  Geometry}.
\newblock Cambridge University Press, 1997.

\bibitem{llt}
A.~Lascoux, B.~Leclerc, and J.-Y. Thibon.
\newblock The plactic monoid.
\newblock In {\em Algebraic Combinatorics on Words by M. Lothaire}, chapter~5.
  Cambridge University Press, 2002.
\newblock
  \url{http://www-igm.univ-mlv.fr/~berstel/Lothaire/AlgCWContents.html}.

\bibitem{ls}
A.~Lascoux and M.-P. Sch\"utzenberger.
\newblock Le mono\"ide plaxique.
\newblock In {\em Non-commutative structures in Alegbra and Geometric
  Combinatorics}, Quaderni de `La ricerca scientifica', n.~109. 1981.

\bibitem{loehr}
N.~A. Loehr.
\newblock Abacus proofs of {S}chur function identities.
\newblock {\em SIAM J. Discrete Math.}, 24, 2010.

\bibitem{macd}
I.~G. Macdonald.
\newblock {\em Symmetric Functions and Hall Polynomials}.
\newblock Oxford University Press, second edition, 1995.

\bibitem{manivel}
L.~Manivel.
\newblock {\em Symmetric Functions, {S}chubert Polynomials and Degeneracy
  Loci}.
\newblock AMS/SMF, 1998.

\bibitem{rtcv}
A.~Prasad.
\newblock {\em Representation Theory: A Combinatorial Viewpoint}.
\newblock Cambridge University Press, 2015.

\bibitem{ec2}
R.~Stanley.
\newblock {\em Enumerative Combinatorics, Volume 2}.
\newblock Cambridge University Press, 1999.

\bibitem{stembridge}
J.~R. Stembridge.
\newblock Nonintersecting paths, pfaffians, and plane partitions.
\newblock {\em Advances in Mathematics}, 83, 1990.

\bibitem{imsc2016}
X.~Viennot.
\newblock An introduction to enumerative, algebraic and bijective
  combinatorics.
\newblock \url{http://www.xavierviennot.org/coursIMSc2016}.
\newblock A course taught at The Institute of Mathematical Sciences, Chennai in
  2016.

\end{thebibliography}
\end{document}